\documentclass{article}

\usepackage{ragged2e} 
\usepackage{amsmath}
\usepackage{amssymb}
\usepackage{mathrsfs}
\usepackage{euscript}
\usepackage{amsthm}
\usepackage{mathtools}
\numberwithin{equation}{section}
\usepackage{bigints}
\DeclarePairedDelimiter{\norm}{\lVert}{\rVert}
\DeclarePairedDelimiter{\abs}{\lvert}{\rvert}
\DeclarePairedDelimiter{\pro}{\langle}{\rangle}

\let\url\texttt
\usepackage[breaklinks,colorlinks]{hyperref}
\usepackage{xcolor}
\usepackage{scalerel,stackengine}
\stackMath
\newcommand\reallywidehat[1]{%
\savestack{\tmpbox}{\stretchto{%
  \scaleto{%
    \scalerel*[\widthof{\ensuremath{#1}}]{\kern-.6pt\bigwedge\kern-.6pt}%
    {\rule[-\textheight/2]{1ex}{\textheight}}
  }{\textheight}%
}{0.5ex}}%
\stackon[1pt]{#1}{\tmpbox}%
}
\usepackage{mathdoc}
\usepackage{amsmath}
\usepackage[colorlinks]{hyperref}
\usepackage{bbm}
\hypersetup{colorlinks=true}
\hypersetup{
    colorlinks=true,
    linkcolor=blue,
    filecolor=blue,      
    urlcolor=blue,
    anchorcolor=black,
    citecolor={blue},
}

\usepackage{color}

\usepackage[left=3cm,right=3cm,top=4cm,bottom=3cm]{geometry}
\usepackage{lipsum}
\usepackage{changepage}
\providecommand\abstractname{Abstract}
\def\abstract{}

\renewenvironment{abstract}{%
  \centering\small
  \textbf\abstractname
  \list{}{\leftmargin2cm \rightmargin\leftmargin}
  \item\relax
}{%
  \endlist \par\bigskip
}

\oddsidemargin=\paperwidth
\advance\oddsidemargin-\textwidth
\evensidemargin=\oddsidemargin

\providecommand{\abs}[1]{\lvert#1\rvert}

\providecommand{\pdfinfo}[1]{}

\theoremstyle{plain}
\newtheorem{thm}{Theorem}[section]

\newtheorem{prop}[thm]{Proposition}
\newtheorem{cor}[thm]{Corollary}
\theoremstyle{definition}

\theoremstyle{remark}
\newtheorem*{rem}{Remark}

\newcommand{\N}{\mathbb{N}}

\newcommand{\R}{\mathbb{R}}

\newcommand{\defi}{\coloneqq}

\begin{document}
\title{\textbf{Radiative Transfer with long-range interactions in the half-space} }

\markboth{RTE  with long-range interactions in the half-space}
         {RTE  with long-range interactions in the half-space}
   

%
%
%
%
%
\author{Ricardo Alonso and Edison  Cuba}

\maketitle

\setcounter{tocdepth}{3}

%
%
%
%
\maketitle

\begin{abstract}
We study the well-posedness and regularity theory for the Radiative Transfer equation in the peaked regime posed in the half-space.  An average lemma for the transport equation in the half-space is stablished and used to generate interior regularity for solutions of the model.  The averaging also shows a fractional regularization gain up to the boundary for the spatial derivatives.
\medskip

\end{abstract}
%


\section{Introduction}\label{int}
Radiative transfer is the physical phenomenon of energy transfer in the form of electromagnetic radiation. The radiative transfer equation (RTE) in the half-space can be written as
\begin{equation}\label{eq:1}
\left\{
\begin{array}{cll}
\partial_t u+\theta\cdot\nabla_x u =\mathcal{I}(u) & \text{in} & (0,T)\times \mathbb{R}^{d}_+\times \mathbb{S}^{d-1},\\
u=u_0 & \text{on} & \{t=0\}\times \mathbb{R}^{d}_+\times \mathbb{S}^{d-1},\\
u=g  & \text{on} & (0,T) \times \{ x_{d} = 0 \}\times\{ \theta_{d}>0 \},
\end{array}\right.
\end{equation}
where $T$ is an arbitrary time, $u = u(t, x, \theta)$ is the radiation intensity distribution in $(0,T)\times \mathbb{R}^{d}_+\times \mathbb{S}^{d-1}$, and where the half-plane has been defined as $\R^{d}_{+}=\{x\,|\, x_{d}>0\}$.  The initial radiation distribution $u_0(x,\theta)$ and the boundary radiation intensity $g(t, \bar{x}, \theta)$ are assumed nonnegative most of the time\footnote{The $L^{2}$ theory of the equation will not require non negativity of the data.  However, \textit{a priori} estimates based on the $L^{1}$ integrability of solutions will require it.}.  We adopt the bar notation $\bar{x}$ to represent points in $\partial\mathbb{R}^{d}_{+}=\{x\,|\,x_{d}=0\}\sim\mathbb{R}^{d-1}$ for any $d\geq3$.

\noindent
The scattering operator is defined as
\begin{equation}
\mathcal{I}(u)\defi \mathcal{I}_{b_s}(u)= \int_{\mathbb{S}^{d-1}}  (u(\theta^{\prime})-u(\theta))b_s(\theta,\theta^{\prime}) \;\mathrm{d}\theta^{\prime}.
\end{equation}
In this work we are interested in the highly forward-peaked regime in the half-space where the angular scattering kernel takes the form
\begin{equation}
b_s(\theta,\theta^{\prime}) =\frac{b(\theta\cdot\theta^{\prime} )}{(1-\theta\cdot\theta^{\prime} )^{\frac{d-1}{2}+s}},\qquad \qquad s\in \left(0,\min \left\lbrace 1,\tfrac{d-1}{2}\right\rbrace\right),
\end{equation}
where  $b(z) \geq 0$ has some smoothness in the neighborhood of $z = 1$. More precisely, we will consider in the sequel its decomposition into two nonnegative components
\begin{equation}\label{b-scattering}
b(z)=b(1)+\tilde{b}(z), \quad \text{where} \quad b(1)>0\quad\text{and}\quad\frac{\tilde{b}(z)}{(1-z)^{1+s}} =: h(z) \in L^{1}(-1,1).
\end{equation}
The weak formulation of this operator is given, for any sufficiently regular test function $\psi$, by
\begin{equation}\label{1.6}
\begin{split}
\MoveEqLeft
\int_{\mathbb{S}^{d-1}}\mathcal{I}(u)(\theta)\psi(\theta)\;\mathrm{d}\theta\defi -\dfrac{1}{2} \int_{\mathbb{S}^{d-1}} \int_{\mathbb{S}^{d-1}} \left(u(\theta^{\prime})-u(\theta)\right)\left(\psi(\theta^{\prime})-\psi(\theta)\right)b_s(\theta,\theta^{\prime})\;\mathrm{d}\theta \;\mathrm{d}\theta^{\prime}.
\end{split}
\end{equation}
The RTE \eqref{eq:1} serves as a model for physical phenomena associated to wave propagation in random media, for example propagation of high frequency waves that are weakly coupled due to heterogeneity, or regimes associated to ``long range'' propagation of waves in weakly heterogeneous media.  Related to the latter, the classical example is the highly forward-peaked regime commonly found in neutron transport, atmospheric radiative transfer, and optical imaging, see \cite{pom, larsen, larsen2}.  The RTE model in the forward-peaked regime is now commonly used for medial imaging inversion since it describes fairly well the propagation of waves through biological tissue in such conditions.

\smallskip
\noindent  
The model \eqref{eq:1} has been studied in the whole domain using slightly different approaches, based on hypo-ellipticity techniques \cite{bouchut}, in the references \cite{alonso,Gomez}.  References treating problems with boundaries are scarce in the context of kinetic equations with singular scattering, however, for the classical kinetic Fokker-Planck equation for absorbing boundary we refer to \cite{Hwang1,Hwang2}.

\smallskip
\noindent
In this contribution, we follow the spirit of the arguments brought in \cite{alonso} and adapt it to consider a flat boundary which is an important case in applications.  The technique is interesting since it can be used for efficient numerical implementation of the equation using spectral methods based on the explicit formulas of Proposition \ref{prop2.1}.   After discussing the generalities of the problem in the remainder of Section 1, we proceed to prove the central result of the paper about hypo-elliptic averaging for the equation with flat and prescribed boundary in Section 2.  In Section 3, \textit{a priori} estimates which lead to interior smoothness of solutions are established.  Finally, in Section 4 with develop the well posedness theory of the model.  The argument is fairly simple and based on the classical literature of the radiative transfer in convex domains.  

\subsection{Definition of solution and basic notation.}
Fix $T>0$.  Take boundary data
$$0\leq \big( u_0,\sqrt{\theta_{d}}\,g \big) \in L^{1}(\mathbb{R}^{d}_{+}\times\mathbb{S}^{d-1})\times L^{2}\big((0,T];L^{2}(\{x_{d}=0\}\times\{\theta_{d}>0\})\big)\,.$$
A nonnegative function
$$u\in L^{\infty}\big([0,T]; L^{1}(\mathbb{R}^{d}_{+}\times\mathbb{S}^{d-1})\big)$$
is a weak solution of the RTE provided \eqref{eq:1} is satisfied in the sense of distributions $\mathcal{D}'\big( [0,T)\times\{x_d\geq0\}\times\mathbb{S}^{d-1} \big)$.

\noindent
In order to simplify notation in treating boundaries we consider the sets
$$\Gamma^{\pm}=\big\{ (\bar{x},\theta)\in \partial\mathbb{R}^{d}_{+}\times \mathbb{S}^{d-1} \, \big| \, \pm \theta_{d}<0\big\}\,,\qquad \Sigma^{T}_{\pm}=(0,T)\times \Gamma^{\pm}\,,$$
and the spaces $L^{2}(\Sigma^{T}_{\pm}\; ;\; | \theta_{d} | \;\mathrm{d}\bar{x}  \; \mathrm{d}\theta \; \mathrm{d}t)$ to be the set of square integrable functions in $\Sigma^{T}_{\pm}$ with respect to the measure $| \theta_{d}  | \;\mathrm{d}\bar{x} \; \mathrm{d}\theta \; \mathrm{d}t$.  Of course, $\text{d}\bar{x}$ represents the standard Lebesgue measure in $\partial\mathbb{R}^{d}_{+}$. In the sequel, we may simply use the shorthand $L^2(\Sigma^{T}_{\pm})$ for such spaces.  In addition, it will be common to use the shorthand notation $L^{2}_{t,x,\theta}$ when the domain of the functions is clear from the context.
\subsection{Representation of the projected scattering operator.}
The analysis in this paper is based on the stereographic projection of the kinetic variable $\theta\in\mathbb{S}^{d-1}$ on the plane $v\in\mathbb{R}^{d-1}$, see \cite{alonso}.  The benefit of such approach is the fact that the geometry of the kinetic space is replaced by Bessel weights which are manageable with Fourier methods and render explicit formulas for numerical implementation.

\noindent
Recall that the stereographic projection $\mathcal{S}:\mathbb{S}^{d-1}\rightarrow \R^{d-1}$ is given by
\begin{equation}\label{stero1}
v_{i} = \mathcal{S}(\theta)_i :=\dfrac{\theta_i}{1-\theta_d},\quad 1\leq i\leq d-1\,.
\end{equation}
Its inverse $\mathcal{J}:\R^{d-1}\rightarrow\mathbb{S}^{d-1} $ is given by
\begin{equation}\label{stero2}
\mathcal{J}_{i}(v)=\dfrac{2v_i}{\pro{v}^{2}},\quad 1\leq i\leq d-1\qquad \text{and}\qquad\mathcal{J}_{d}(v)=\dfrac{\abs{v}^{2}-1}{\pro{v}^{2}},
\end{equation}
where $\pro{\cdot}\defi \sqrt{\abs{\cdot}^{2}+1}$ is the Japanese bracket.  The Jacobian of such transformations can be computed as
$$\mathrm{d}v=\dfrac{\mathrm{d}\theta}{(1-\theta_d)^{d-1}}\qquad \text{and} \qquad \mathrm{d}\theta=\dfrac{2^{d-1}\;\mathrm{d}v}{\pro{v}^{2(d-1)}}\,.$$
Using the shorthanded notation $\theta=\mathcal{J}(v)$ and $\theta^{\prime}=\mathcal{J}(v^{\prime})$ we obtain that
$$1-\theta\cdot \theta^{\prime}=2\,\dfrac{\abs{v-v^{\prime}}^{2}}{\pro{v}^{2}\pro{v^{\prime}}^{2}}.$$
In the sequel we use the shorthand $u_{\mathcal{J}} :=u \circ \mathcal{J}:\mathbb{R}^{d-1}\rightarrow \mathbb{S}^{d-1}\rightarrow\mathbb{R}$ for pull back functions, and with a capital letter we introduce the function $U_{\mathcal{J}}:= \frac{u_{\mathcal{J}}}{\pro{\cdot}^{d-1-2s}}$ that will be important along the document.
\begin{prop}\label{prop2.1}
For any sufficiently regular function $u$ in the sphere the stereographic projection of the operator $\mathcal{I}_{b(1)}$ is given by
\begin{equation}\label{eq2.1b}
\begin{split}
\dfrac{\left[ \mathcal{I}_{b(1)}(u)\right]_{\mathcal{J}}}{\pro{\cdot}^{d-1+2s}}& = \dfrac{2^{\frac{d-1}{2}-s}b(1)}{c_{d-1,s}}\left(-(-\Delta_v)^{s}U_{\mathcal{J}}+u_{\mathcal{J}}\,(-\Delta_v)^{s}\bigg(\dfrac{1}{\pro{\cdot}^{d-1-2s}}\bigg)\right)
\\
 & = \dfrac{2^{\frac{d-1}{2}-s}\,b(1)}{c_{d-1,s}}\left(-(-\Delta_v)^{s}U_{\mathcal{J}} + c_{d,s}\,\dfrac{u_{\mathcal{J}}}{\pro{\cdot}^{d-1+2s}}\right).
\end{split}
\end{equation}
As a consequence, we have that
\begin{equation}\label{eq10}
  \begin{split}
    \MoveEqLeft
  \dfrac{1}{b(1)}\int_{\mathbb{S}^{d-1}}\mathcal{I}_{b(1)}(u)(\theta)\,\overline{u(\theta)}\,\mathrm{d}\theta
=  -c_{d,s}\,\norm{(-\Delta_v)^{s/2}U_{\mathcal{J}}}^{2}_{L^{2}(\R^{d-1})}+C_{d,s}\norm{u}^{2}_{L^{2}(\mathbb{S}^{d-1})},
       \end{split}
\end{equation}
for some explicit positive constants $c_{d,s}$ and $C_{d,s}$. Furthermore,  defining the differential operator $(-\Delta_{\theta})^{s}$ acting on functions defined on the sphere by the formula
\begin{equation}\label{2.4}
\big[ (-\Delta_{\theta})^{s}u \big]_{\mathcal{J}}\defi \pro{\cdot}^{d-1+2s}\,(-\Delta_{v})^{s}U_{\mathcal{J}},
\end{equation}
the scattering operator, $\mathcal{I}_{b_{s}}=\mathcal{I}_{b(1)}+\mathcal{I}_{h}$, simply writes as the sum of a singular part and $L^{2}_{\theta}-$ bounded part
\begin{equation}\label{2.5}
\mathcal{I}_{b_s}=-D\,(-\Delta_{\theta})^{s}+c_{s,d}\,I+\mathcal{I}_{h},
\end{equation}
where $D=2^{\frac{d-1}{2}-s}\frac{b(1)}{c_{d-1,s}}$ is the diffusion constant.
\end{prop}
\begin{proof}
For the details of the proof see \cite{alonso}.
\end{proof}
\subsection{ Functional spaces.}
Recalling the notation $U_{\mathcal{J}}:= \frac{u_{\mathcal{J}}}{\pro{v}^{d-1-2s}}$, the fractional Sobolev spaces $H^{s}(\mathbb{S}^{d-1})$ is defined as
$$H^{s}_{\theta}\defi \big\{ u \in L^{2}_{\theta}: (-\Delta_v)^{s/2}U_{\mathcal{J}}\in L^{2}_{v}\big\},\qquad s\in (0,1),$$
endowed with the inner product
$$\pro{u,w}_{H^{s}_{\theta}}\defi \pro{(-\Delta_v)^{s/2}U_{\mathcal{J}},(-\Delta_v)^{s/2}W_{\mathcal{J}}}_{L^{2}(\R^{d-1})}\,.$$

\noindent The following useful representation of the inner product norm in $H^{s}(\mathbb{S}^{d-1})$ which follows directly from \eqref{eq10} and \eqref{1.6} will be important in the next sections,
\begin{equation}\label{2.12}
\begin{split}
\norm{u}^{2}_{H^{s}_{\theta}(\mathbb{S}^{d-1})} &\sim -\int_{\mathbb{S}^{d-1}}\mathcal{I}_{b(1)}(u)(\theta)\,\overline{u(\theta)}\,\mathrm{d}\theta + \int_{\mathbb{S}^{d-1}}\abs{u(\theta)}^{2}\;\mathrm{d}\theta\\
&\sim \int_{\mathbb{S}^{d-1}}\int_{\mathbb{S}^{d-1}}\dfrac{(u(\theta^{\prime})-u(\theta))^{2}}{\abs{\theta^{\prime}-\theta}^{d-1+2s}}\;\mathrm{d}\theta^{\prime}\;\mathrm{d}\theta+\norm{u}^{2}_{L^{2}(\mathbb{S}^{d-1})}.
\end{split}
\end{equation}
where, for the second equivalence, we used that $2(1-\theta\cdot \theta^{\prime})=\abs{\theta^{\prime}-\theta}^{2}$ valid for any two unitary vectors.  We also have, by a straightforward computation, the Sobolev embedding $H^{s}_{\theta} \hookrightarrow L^{p_{s}}_{\theta}$
\begin{equation}\label{2.14}
\norm{u}_{L^{p_s}_{\theta}} \leq C_{d,s}\norm{u}_{H^{s}_{\theta}}\,,\qquad \frac{1}{p_s}=\frac{1}{2}-\frac{s}{d-1}\,.
\end{equation}

\subsection{Natural a priori energy estimates.}
Assume the existence of a sufficiently smooth nonnegative solution $u$ to the RTE problem (\ref{eq:1}).  Direct integration in $(x,\theta)\in\mathbb{R}^{d}_{+}\times\mathbb{S}^{d-1}$ and time $0\leq t^{\prime} \leq t \leq T$, together with the divergence theorem and the fact that $\mathcal{I}$ is a mass conceving operator $\int_{ \mathbb{S}^{d-1}} \mathcal{I}(u)\, \text{d}\theta = 0$ gives that
\begin{equation}\label{eq-mass}
\begin{split}
&\int_{\R^{d}_{+}}\int_{\mathbb{S}^{d-1}} u(t,x,\theta)\; \mathrm{d}\theta\; \mathrm{d}x + \int^{t}_{t^{\prime}}\int_{\Gamma^{+}} u\, ( \theta\cdot n(\bar{x}) )\; \mathrm{d}\theta\; \mathrm{d}\bar{x}\; \mathrm{d}\tau\\
&\qquad= \int_{\R^{d}_{+}}\int_{\mathbb{S}^{d-1}} u(t^{\prime},x,\theta)\; \mathrm{d}\theta\; \mathrm{d}x + \int^{t}_{t^{\prime}}\int_{\Gamma^{-}}g \big| \theta\cdot n(\bar{x}) \big|\; \mathrm{d}\theta\; \mathrm{d}\bar{x}\; \mathrm{d}\tau\,.
\end{split}
\end{equation}
This identity describes the mass in the system.  The boundary terms, from left to right, represent the out flux and in flux of mass through the boundary.

\noindent
Now, in order to obtain the description of the energy multiply the RTE equation by $u$ and integrate in the variables $(x,\theta)\in\mathbb{R}^{d}_{+}\times\mathbb{S}^{d-1}$ to obtain that
\begin{align*}
&\tfrac{1}{2}\frac{{\rm d}}{ {\rm d }t}\int_{\mathbb{R}^{d}_{+}}\int_{\mathbb{S}^{d-1}} u^{2}\; \mathrm{d}\theta\; \mathrm{d}x+\tfrac{1}{2}\int_{\mathbb{R}^{d}_{+}}\int_{\mathbb{S}^{d-1}}\theta\cdot \nabla_x u^{2}\; \mathrm{d}\theta\; \mathrm{d}x\\
&=\tfrac{1}{2}\frac{{\rm d}}{ {\rm d }t}\int_{\mathbb{R}^{d}_{+}}\int_{\mathbb{S}^{d-1}} u^{2}\; \mathrm{d}\theta\; \mathrm{d}x+\tfrac{1}{2}\int_{\partial\mathbb{R}^{d}_{+}}\int_{\mathbb{S}^{d-1}} u^{2} \left(\theta\cdot n(\bar{x}) \right)\; \mathrm{d}\theta\; \mathrm{d}\bar{x}=\int_{\mathbb{R}^{d}_{+}}\int_{\mathbb{S}^{d-1}}\mathcal{I}(u)\,u\, \mathrm{d}\theta\; \mathrm{d}x\,,
\end{align*}
where we used, again, the divergence theorem in the second step.  Therefore,
\begin{equation}\label{eq 2.3}
\begin{split}
\tfrac{1}{2}\frac{{\rm d}}{ {\rm d }t}&\int_{\mathbb{R}^{d}_{+}}\int_{\mathbb{S}^{d-1}} u^{2}\; \mathrm{d}\theta\; \mathrm{d}x+\tfrac{1}{2}\int_{\Gamma^{+}} u^{2}\left(\theta\cdot n(\bar{x})\right)\; \mathrm{d}\theta\; \mathrm{d}\bar{x}\\
&\qquad=\tfrac{1}{2}\int_{\Gamma^{-}}g^{2}\big| \theta\cdot n(\bar{x})\big|\; \mathrm{d}\theta\; \mathrm{d}\bar{x}+\int_{\mathbb{R}^{d}_{+}}\int_{\mathbb{S}^{d-1}}\mathcal{I}(u)\, u\, \mathrm{d}\theta\; \mathrm{d}x.
\end{split}
\end{equation}
Since $\int_{\mathbb{S}^{d-1}}\mathcal{I}(u)\, u\, \mathrm{d}\theta\leq0$, we can integrate in time  $0<t^{\prime}\leq t<T$ to conclude that
\begin{equation}\label{eq2.3-1}
\begin{split}
&\tfrac{1}{2}\int_{\R^{d}_{+}} \| u(t) \|^{2}_{L^{2}_{\theta}}\; \mathrm{d}x + \tfrac{1}{2}\int^{t}_{t^{\prime}}\int_{\Gamma^{+}} u^{2} ( \theta\cdot n(\bar{x}) )\; \mathrm{d}\theta\; \mathrm{d}\bar{x}\; \mathrm{d}\tau\\
&\qquad\qquad \leq \tfrac{1}{2}\int_{\R^{d}_{+}}\int_{\mathbb{S}^{d-1}} \| u(t^{\prime})\|^{2}_{L^{2}_{\theta}}\;\mathrm{d}x + \tfrac{1}{2}\int^{t}_{t^{\prime}}\int_{\Gamma^{-}} g^{2}\big| \theta\cdot n(\bar{x}) \big|\; \mathrm{d}\theta\; \mathrm{d}\bar{x}\; \mathrm{d}\tau\,.
\end{split}
\end{equation}
This estimate can be upgraded to add the diffusion term in the scattering angle using relation \eqref{2.12}.  We are led to
\begin{equation}\label{eq2.4}
\begin{split}
&\tfrac{1}{2}\int_{\R^{d}_{+}} \| u(t) \|^{2}_{L^{2}_{\theta}}\; \mathrm{d}x+D_0\int^{t}_{t^{\prime}}\int_{\R^{d}_{+}}\norm{u}^{2}_{H^{s}_{\theta}} \; \mathrm{d}x\; \mathrm{d}\tau \\
&\qquad+ \tfrac{1}{2}\int^{t}_{t^{\prime}}\int_{\Gamma^{+}} u^{2} ( \theta\cdot n(\bar{x}) )\; \mathrm{d}\theta\; \mathrm{d}\bar{x}\; \mathrm{d}\tau \leq \tfrac{1}{2}\int_{\R^{d}_{+}}\int_{\mathbb{S}^{d-1}} \| u(t^{\prime})\|^{2}_{L^{2}_{\theta}}\;\mathrm{d}x\\
&\qquad\qquad +D_1\int^{t}_{t^{\prime}}\int_{\R^{d}_{+}} \| u(\tau) \|^{2}_{L^{2}_{\theta}} \; \mathrm{d}x\; \mathrm{d}\tau  + \tfrac{1}{2}\int^{t}_{t^{\prime}}\int_{\Gamma^{-}} g^{2}\big| \theta\cdot n(\bar{x}) \big|\; \mathrm{d}\theta\; \mathrm{d}\bar{x}\; \mathrm{d}\tau\,.
\end{split}
\end{equation}
Here $D_0$ depends on $d, s$ and $b(1)$, while $D_1$ depends on $d, s, b(1)$ and the integrable scattering kernel $h(\cdot)$.  Estimate \eqref{eq2.4} shows the diffusion nature of the equation in the scattering variable $\theta$.  We will complete, using the Proposition \ref{cor1} below, such estimate to include the spatial diffusion nature of the model as well.  
\section{Averaging lemma in the half-space}
In this section, we give  a regularization mechanism in the RTE. It is related to the fact that the diffusion in the kinetic variable $\theta$ is propagated to the spatial variable by means of the advection operator $\theta\cdot\nabla_{x}$.  We follow the framework developed in \cite{bouchut,alonso} and adapt it to the fact that we are considering half-space on the spatial variable.

\smallskip
\noindent
In the sequel, the Fourier transform in time and spatial variables for a suitable function $\varphi(t,x)$ defined in $(0,\infty)\times\mathbb{R}^{d}_{+}$ is given by
\begin{equation}\label{FT}
\widehat{\varphi}(w,k) = \mathcal{F}_{t,x}\{\varphi\}(w,k) = \int^{\infty}_{0}\int_{\mathbb{R}^{d}_{+}} \varphi(t,x)\, \text{e}^{-\text{i} w\,t} \text{e}^{-\text{i} k\cdot x}\text{d}x\,\text{d}t\,, \quad (w,k)\in\R\times\R^{d}\,.
\end{equation}
The partial Fourier transforms in time $\mathcal{F}_{t}\{\cdot\}$ and space $\mathcal{F}_{x}\{\cdot\}$ for $\varphi(t,x)$ are defined in obvious manner.  The fractional differentiation in the spatial variable for a sufficiently smooth function $\varphi(x)$ with domain in $\R^{d}_{+}$ is defined through its Fourier transform
\begin{equation}\label{SpatialDiff}
\mathcal{F}_{x}\big\{ (-\Delta_{x})^{\frac{s}{s}}\varphi \big\}(k) = |k|^{s}\int_{\mathbb{R}^{d}_{+}} \varphi(x)\, \text{e}^{-\text{i} k\cdot x}\text{d}x\,,\qquad k\in\R^{d}\,,\quad s>0\,.
\end{equation}
Note that this definition agrees with the classical one using the extension of $\varphi$ by zero in $\{x\,|\, x_{d}\leq 0\}$.  As a consequence, $(-\Delta_{x})^{\frac{s}{s}}\varphi$ is a tempered distribution defined in $\R^{d}$ and have the information of the trace of $\varphi$ on $\partial\mathbb{R}^{+}_{+}$ encoded.
\begin{thm}\label{thm1}
Fix $d\geq 3$ and boundary $\theta_{d}\, g\in L^{2}((t_0,t_1)\times\Gamma^{-})$.  Assume that $u\in \mathcal{C}([t_0,t_1);L^{2}({\mathbb{R}^{d}_{+}\times\mathbb{S}^{d-1}})$ solves the RTE on the half-space (\ref{eq:1}) for  $t\in(t_0,t_1)$. Then, for any $s \in (0, 1)$, there exists a constant $C \defi C(d, s)$ such that 
\begin{equation}\label{eqthm}
\begin{aligned}
\Vert (-\Delta_x)^{s_0/2} u \Vert_{L^{2}((t_0,t_1)\times \R^{d}\times\mathbb{S}^{d-1})} \leq C \left(\Vert u(t_0) \Vert_{L^{2}( \R^{d}_+\times\mathbb{S}^{d-1})} +\Vert u \Vert_{L^{2}((t_0,t_1)\times \R^{d}_+\times\mathbb{S}^{d-1})} \right.  \\
\left.+\Vert (-\Delta_v)^{s}U_{\mathcal{J}} \Vert_{L^{2}((t_0,t_1)\times \R^{d}_+\times\mathbb{R}^{d-1})} +\norm{\theta_{d}\,g}_{L^{2}((t_0,t_1)\times\Gamma^{-})}\right), \quad s_0=\frac{s/8}{2s+1}.
\end{aligned}
\end{equation}
\end{thm}
\begin{proof}
Start with an approximation of the identity in the sphere $\{\rho_{\epsilon}\}_{\epsilon>0}$ defined through an smooth function   $\rho \in \mathcal{C}(-1,1)$ satisfying the properties
\begin{equation}\label{eq.3}
\int^{1}_{-1}  \rho(z)z^{\frac{d-3}{2}} \;\mathrm{d}z=1, \qquad 0<\rho(z)\lesssim \frac{1}{z^{\frac{d-1}{2}+s}}.
\end{equation}
Introduce the quantity
\begin{equation}\label{eq4}
C_\epsilon=\vert \mathbb{S}^{d-2}\vert \int^{1}_{-1}  \rho(z)z^{\frac{d-3}{2}}(2-\epsilon z)^{\frac{d-3}{2}} \;\mathrm{d}z, \qquad \epsilon \in \left( 0,1\right].
\end{equation}
and note that $\inf_{\epsilon\in(0,1]} C_{\epsilon}>0$. Thus, define the approximation of the identity as
\begin{equation}\label{eq5}
\rho_{\epsilon}(z)=\dfrac{1}{C_{\epsilon}\epsilon^{\frac{d-1}{2}}}\rho \left(\dfrac{z}{\epsilon}\right).
\end{equation}
We can see that
\begin{equation}\label{eq6}
\int_{\mathbb{S}^{d-1}}  \rho_{\epsilon}(1-\theta\cdot \theta ^{\prime}) \;\mathrm{d}\theta^{\prime}=1, \qquad \epsilon>0.
\end{equation}
We understand the convolution in the sphere, for any real function $\psi$ defined on the sphere, as
\begin{equation}\label{eq7}
(\rho\star \psi)(\theta)=\int_{\mathbb{S}^{d-1}}  \rho(1-\theta\cdot \theta ^{\prime})\psi(\theta^{\prime}) \;\mathrm{d}\theta^{\prime}.
\end{equation}
Now, consider a sufficiently smooth solution $u$ of the RTE on the half-space, in the interval $[t_0, t_1]$ for any $0 < t_0 <
t_1 < \infty$.  The Fourier transform of $\partial_t u$ can be computed as
\begin{equation*}
\reallywidehat{\partial_t u}(w,k,\theta)= - \mathrm{e}^{-\mathrm{i}wt_0}\mathcal{F}_{x}\{u\}(t_0,k,\theta) + \text{i}w\,\widehat{u}(w,k,\theta)\,,
\end{equation*}
where the boundary component at $t_1$ is disregarded by the causality of the equation.  In the same spirit we can compute the Fourier transform of $\theta\cdot \nabla_x  u$.  To understand the spectral transformation one considers the problem
\begin{equation*}
\Bigg\{
\begin{array}{cl}
\theta\cdot\nabla_{x}u = f & \text{in}\;\; \mathbb{R}^{d}_{+}\times\mathbb{S}^{d-1}\,, \\
u = g   & \text{on} \;\;\Gamma^{-}\,.
\end{array}
\end{equation*}
The characteristics $x+t\theta$ imply that for $\theta_{d}>0$, both the boundary $g$ and the interior values $f=\theta\cdot\nabla_{x}u$ contribute to $u$, while for $\{\theta_{d}<0\}$ only $f$ contributes to $u$.  In particular, the value of $u$ at $\Gamma^{+}$ is fully determined by the knowledge of $\theta\cdot\nabla_{x}u$ in $\mathbb{R}^{d}_{+}\times\mathbb{S}^{d-1}$, see \cite{Lax-Phillips}.

\smallskip
\noindent
Keeping this in mind, set $x=(\bar{x},x_d)$ with $\bar{x}=(x_1,x_2,\cdots\cdots,x_{d-1})\in \R^{d-1}$ and $x_{d}>0$.  In our coordinate system we will consider $\R^{d}_{+}=\{x\, |\, x_{d}>0\}$.  For the spatial Fourier variable we perform a similar decomposition  $k=(\bar{k},k_{d})$ with $\bar{k}\in\mathbb{R}^{d-1}$ and $k_{d}\in\mathbb{R}$.  Then 
\begin{equation*}
\begin{split}
\MoveEqLeft[2.5]
\mathcal{F}_{x}\big\{ \theta\cdot \nabla_x  u \big\}(t,k,\theta) = \int_{\mathbb{R}^{d}_{+}} \theta\cdot \nabla_x  u\;\ \mathrm{e}^{-\mathrm{i}k\cdot x} \,\mathrm{d}x \\
& = \int_{0}^{+\infty}   \int_{\mathbb{R}^{d-1}}\big( \bar{\theta}\cdot \nabla_{\bar{x}}u +\theta_d\cdot \partial_{x_d} u\big)\;\ \mathrm{e}^{-\mathrm{i}\bar{x}\cdot \bar{k}} \,\mathrm{d}\bar{x}\;\ \mathrm{e}^{-\mathrm{i}x_d\, {k}_d} \,\mathrm{d}{x}_d\\
& =  \int_{0}^{+\infty} \Big(  \int_{\mathbb{R}^{d-1}}\bar{\theta}\cdot \nabla_{\bar{x}}u\;\ \mathrm{e}^{-\mathrm{i}\bar{x}\cdot \bar{k}}\, \text{d}\bar{x} +  \int_{\mathbb{R}^{d-1}}\theta_d\cdot \partial_{x_d} u\;\ \mathrm{e}^{-\mathrm{i}\bar{x}\cdot \bar{k}} \,\mathrm{d}\bar{x} \, \Big) \mathrm{e}^{-\mathrm{i}x_d\, {k}_d} \,\mathrm{d}{x}_d \\
&=\int_{0}^{+\infty} \Big( (\mathrm{i}\, \bar{\theta}\cdot \bar{k})\,\mathcal{F}_{\bar{x}}\{u\}(t,\bar{k},x_d,\theta) +\theta_d\cdot \mathcal{F}_{\bar{x}}\{ \partial_{x_d} u\}(t,\bar{k},x_d,\theta)\, \Big) \mathrm{e}^{-\mathrm{i}x_d\, {k}_d} \,\mathrm{d}{x}_d\\
&=(\mathrm{i}\, \bar{\theta}\cdot \bar{k})\mathcal{F}_{x}\{u\}(t,k,\theta) + \theta_d\int_{0}^{+\infty} \partial_{x_d}\mathcal{F}_{\bar{x}}\{u\}(t,\bar{k},x_d,\theta)\;\  \mathrm{e}^{-\mathrm{i}x_d\, {k}_d} \,\mathrm{d}{x}_d\\
&= (\mathrm{i}\, \bar{\theta}\cdot \bar{k})\; \mathcal{F}_{x}\{u\}(t,k,\theta)+\theta_d \bigg[-\mathcal{F}_{\bar{x}}\{u\}(t,\bar{k},0,\theta)+\mathrm{i}k_d\, \mathcal{F}_{x}\{u\}(w,k,\theta) \bigg]\\[4pt]
&= (\mathrm{i}\, \theta \cdot k)\;\mathcal{F}_{x}\{u\}(t,k,\theta) - \mathcal{F}_{\bar{x}}\{G\}(t,\bar{k},\theta),
\end{split}
\end{equation*}
where for the boundary $\{x\,|\,x_{d}=0\}$, we introduced
\begin{equation}\label{G-boundary}
G(t,\bar{x},\theta):=\Bigg\{
\begin{array}{cll}
0 & \text{if} & \theta_{d}<0\,,\\
\theta_{d}\, g(t,\bar{x},\theta) & \text{if} & \theta_{d}>0\,,
\end{array}
\end{equation}
because $u=g$ on $\Sigma^{T}_{-}$.  Meanwhile, $G$ vanishes in $\Sigma^{T}_{+}$ because the term $\mathrm{i}\, \theta \cdot k\,\mathcal{F}_{x}\{u\}$ uniquely defines $u$ at $\Gamma^{+}$.   As a consequence,
\begin{equation*}
\widehat{\theta\cdot \nabla_x  u}(w,k,\theta) = (\mathrm{i}\, \theta \cdot k)\;\widehat{u}(w,k,\theta) - \widehat{G}(w,\bar{k},\theta)\,,
\end{equation*}
where $\widehat{G}=\mathcal{F}_{t,\bar{x}}\{G\}$.  Overall, we conclude that Fourier transform of (\ref{eq:1}) is given by
\begin{equation}\label{eq8}
\mathrm{i}(w + \theta \cdot k)\; \widehat{u}(w, k, \theta) = \mathcal{I}(\widehat{u})(w, k, \theta) +\widehat{u}(t_0, k, \theta)\; \mathrm{e}^{-iwt_0} + \widehat{G}(w,\bar{k},\theta).
 \end{equation}

\noindent  A key step in the proof is to decompose $\widehat{u}$, for any fixed $(w, k)$, as
\begin{equation}\label{eq9}
\widehat{u}(w,k,\theta)=(\rho \star \widehat{u})(w,k,\theta)+\big[\widehat{u}(w,k,\theta)-(\rho \star \widehat{u})(w,k,\theta)\big].
\end{equation}
From (\ref{eq.3})-(\ref{eq8}) and Proposition \ref{prop2.1}, the error can be estimated similarly as in Theorem 3.2 on \cite{alonso} by
\begin{align}\label{3.26}
\begin{split}
\Vert  \widehat{u}(w,k,\cdot) - &(\rho_\epsilon\star \widehat{u})(w,k,\cdot)\Vert^{2}_{L^{2}(\mathbb{S}^{d-1})}\\
     &=  \int_{\mathbb{S}^{d-1}}\abs*{\;\ \int_{\mathbb{S}^{d-1}}\rho_{\epsilon}(1-\theta \cdot \theta^{\prime})\;(\widehat{u}(w,k,\theta)-\widehat{u}(w,k,\theta^{\prime}))\; \mathrm{d}\theta^{\prime}}^{2}\mathrm{d}\theta\\
     &\leq \int_{\mathbb{S }^{d-1}}\; \int_{\mathbb{S}^{d-1}}\rho_{\epsilon}(1-\theta \cdot \theta^{\prime})\; \abs*{\widehat{u}(w,k,\theta)-\widehat{u}(w,k,\theta^{\prime})}^{2}\; \mathrm{d}\theta^{\prime}\; \mathrm{d}\theta\\ 
       &\lesssim \dfrac{\epsilon^{s}}{^{C_{\epsilon}}} \; \int_{\mathbb{S }^{d-1}}\; \int_{\mathbb{S}^{d-1}}\; \dfrac{\abs*{\widehat{u}(w,k,\theta)-\widehat{u}(w,k,\theta^{\prime})}^{2}}{(1-\theta \cdot \theta^{\prime})^{\frac{d-1}{2}+s}}\; \mathrm{d}\theta^{\prime}\; \mathrm{d}\theta\\  
       & =  -\dfrac{2\epsilon^{s}}{^{C_{\epsilon}b(1)}} \int_{\mathbb{S }^{d-1}}\; \mathcal{I}_{b(1)}(\widehat{u})(w,k,\theta)\overline{\widehat{u}(w,k,\theta)}\;\mathrm{d}\theta\\
       &\lesssim  \dfrac{\epsilon^{s}}{^{C_{\epsilon}}} \; \norm*{ (-\Delta_v)^{s/2}\; \reallywidehat{U_{\mathcal{J}}}(w,k,\cdot)  }^{2}_{L^{2}(\mathbb{S}^{d-1})}.
       \end{split}
\end{align} 
Now, we  estimate the term $\rho \star \widehat{u}$ in (\ref{eq9}) for each fixed $(w, k)$.  Using  (\ref{eq8})
\begin{equation}\label{eq11}
\widehat{u}=\dfrac{\lambda \widehat{u}+\mathcal{I}(\widehat{u})+\widehat{u}(t_0,k,\theta)\mathrm{e}^{-\mathrm{i}w t_0} + \widehat{G}(w,\bar{k},\theta)}{\lambda +\mathrm{i}(w+\theta \cdot k)},
\end{equation}
where $\lambda > 0$ is an interpolation parameter depending only on $\abs{k}$ (the parameter $\epsilon$ will depend only on $\abs{k}$ as well). Formulas  (\ref{2.5}) and (\ref{eq11}) lead to
\begin{equation}
\begin{split}
(\rho \star \widehat{u})(w,k,\theta)& = \int_{\mathbb{S}^{d-1}}\rho_{\epsilon}(1-\theta \cdot \theta^{\prime})\dfrac{\lambda \widehat{u}(w,k,\theta^{\prime
})+\mathcal{I}(\widehat{u})(w,k,\theta^{\prime
})+\widehat{u}(t_0,k,\theta^{\prime})\; \mathrm{e}^{-\mathrm{i}wt_0}}{\lambda +\mathrm{i}(w+k\cdot \theta^{\prime
})} \,\mathrm{d}\theta^{\prime}
  \\
 & = \int_{\mathbb{S}^{d-}1}\rho_{\epsilon}(1-\theta \cdot \theta^{\prime})\dfrac{ \widehat{u}(w,k,\theta^{\prime
})+\frac{1}{\lambda}\; \mathcal{K}(\widehat{u})(w,k,\theta^{\prime
})}{1 +\mathrm{i}(w+k\cdot \theta^{\prime
})/\lambda} \,\mathrm{d}\theta^{\prime} 
  \\
 & \quad - \dfrac{D}{\lambda}\int_{\mathbb{S}^{d-}1}\rho_{\epsilon}(1-\theta \cdot \theta^{\prime})\dfrac{(-\Delta_{\theta^{\prime}})^{s}\; \widehat{u}(w,k,\theta^{\prime})}{1 +\mathrm{i}(w+k\cdot \theta^{\prime})/\lambda} \,\mathrm{d}\theta^{\prime}  
 \\
  & \quad + \dfrac{1}{\lambda}\int_{\mathbb{S}^{d-}1}\rho_{\epsilon}(1-\theta \cdot \theta^{\prime})\dfrac{\widehat{u}(t_0,k,\theta)\; \mathrm{e}^{-\mathrm{i}wt_0}}{1 +\mathrm{i}(w+k\cdot \theta^{\prime})/\lambda} \,\mathrm{d}\theta^{\prime}
  \\
    & \quad + \dfrac{1}{\lambda}\int_{\mathbb{S}^{d-}1}\rho_{\epsilon}(1-\theta \cdot \theta^{\prime})\dfrac{\widehat{G}(w,k,\theta)}{1 +\mathrm{i}(w+k\cdot \theta^{\prime})/\lambda} \,\mathrm{d}\theta^{\prime}\triangleq T_1+T_2+T_3+T_4,
  \end{split}
\end{equation}
where $\mathcal{K}\defi c_{s,d}\textbf{1}+\mathcal{I}_{h}$ is the bounded part of $\mathcal{I}$.

\noindent
The terms $T_{i}$ for $i=1,2,3$ have been estimated in \cite{alonso} in formula (3.33) for $T_{1}$, formulas (3.40), (3.41), (3.44) for $T_{2}$ and formula (3.47) for $T_{3}$.  Let us write the overall result and refer to \cite{alonso} for the details.  

\noindent The term $T_1$ is estimated as
\begin{equation}\label{eq16}
\begin{split}
\MoveEqLeft
\norm{T_1(w,k,\cdot)}_{L^{2}(\mathbb{S}^{d-1})}\\
&\leq C\left( \dfrac{1}{\sqrt{\epsilon}}\sqrt{\dfrac{\lambda}{\abs{k}}}\right)^{\frac{1}{2}}\left( \norm{\widehat{u}(w,k,\cdot)}_{L^{2}(\mathbb{S}^{d-1})}+\frac{1}{\lambda}\norm{\mathcal{K}(\widehat{u})(w,k,\cdot)}_{L^{2}(\mathbb{S}^{d-1})}\right).
\end{split}
\end{equation}
The term $T_{2}$ satisfies
\begin{align}\label{eq28}
\begin{split}
\norm*{T_{2}(w,k,\cdot)}_{L^{p}(\mathbb{S}^{d-1})}&\leq \dfrac{D}{\lambda \epsilon}\left(\dfrac{1}{\sqrt{\epsilon}}\sqrt{\dfrac{\lambda}{\abs{k}}}\right)^{\frac{1}{q}}\norm*{(-\Delta_{v})^{s/2}\reallywidehat{U_{\mathcal{J}}}(w,k,\cdot)}_{L^{2}(\R^{d-1})}\\
&\quad+ \dfrac{D\abs{k}}{\lambda^{2}}\left(\dfrac{1}{\sqrt{\epsilon}}\sqrt{\dfrac{\lambda}{\abs{k}}}\right)^{\frac{1}{q}}\norm*{(-\Delta_{v})^{s/2}\reallywidehat{U_{\mathcal{J}}}(w,k,\cdot)}_{L^{2}(\R^{d-1})}\\
&\qquad+ \dfrac{D}{\lambda\sqrt{\epsilon}}\left(\dfrac{1}{\sqrt{\epsilon}}\sqrt{\dfrac{\lambda}{\abs{k}}}\right)^{\frac{2-q}{q}}\norm*{(-\Delta_{v})^{s/2}\reallywidehat{U_{\mathcal{J}}}(w,k,\cdot)}_{L^{2}(\R^{d-1})}.
\end{split}
\end{align}
where $p:=p(s,d)$ and $q:=q(s,d)$, which are convex dual, are given by the formulas
\begin{equation*}
\frac{1}{2}-\frac{1-s}{d-1}=\frac{1}{p}\quad \text{and} \quad \frac{1}{2}+\frac{1-s}{d-1}=\frac{1}{q}\,,\quad d\geq3\,.
\end{equation*}
The term $T_{3}$ is estimates as
\begin{equation}\label{eq31}
\norm*{T_{3}(\cdot,k,\cdot)}_{L^{2}(\R\times \mathbb{S}^{d-1})}\leq \dfrac{C}{\lambda}\left(\dfrac{1}{\sqrt{\epsilon}}\sqrt{\dfrac{\lambda}{\abs{k}}}\right)^{\frac{1}{2}}\norm*{\widehat{u}(t_0,k, \cdot)}_{L^{2}(\mathbb{S}^{d-1})}.
\end{equation} 
Finally, we estimate the boundary term $T_4$ as follows.  Note that 
\begin{equation}\label{eq32}
\begin{split}
\abs{T_4(w,k,\theta)}&=\dfrac{1}{\lambda}\int_{\mathbb{S}^{d-1}}\rho_{\epsilon}(1-\theta\cdot \theta^{\prime})\dfrac{  \abs{\widehat{G}(w,\bar{k},\theta^{\prime})}}{\abs{1+i(w+k\cdot\theta^{\prime})/\lambda}}\mathrm{d}\theta^{\prime}
\\
&\leq \dfrac{1}{\lambda}\left(\int_{\mathbb{S}^{d-1}}\dfrac{\rho_{\epsilon}(1-\theta\cdot \theta^{\prime})}{\abs{1+i(w+k\cdot\theta^{\prime})/\lambda}^{2}}\mathrm{d}\theta^{\prime}\right)^{\frac{1}{2}}
\\
& \qquad \times  \left(\int_{\mathbb{S}^{d-1}}\rho_{\epsilon}(1-\theta\cdot \theta^{\prime})\; \abs{\widehat{G}(w,\bar{k},\theta^{\prime})}^{2}\; \mathrm{d}\theta^{\prime}\right)^{\frac{1}{2}}.
\end{split}
\end{equation}
Estimate \cite[(3.32)]{alonso} reads
\begin{equation}\label{eq33}
\int_{\mathbb{S}^{d-1}}\dfrac{\rho_{\epsilon}(1-\theta\cdot \theta^{\prime})}{\abs{1+i(w+k\cdot\theta^{\prime})/\lambda}^{2}}\mathrm{d}\theta^{\prime} \lesssim\dfrac{1}{C\sqrt{\epsilon}}\sqrt{\dfrac{\lambda}{\abs{k}}}.
\end{equation}
Hence, using estimate (\ref{eq33}) in (\ref{eq32}) and integrating in the variables $w$ and $\theta$ we have that
\begin{equation*}
\begin{split}
&\bigg(\int_{\mathbb{R}}\int_{\mathbb{S}^{d-1}}\; \abs{T_4(w,k,\theta)}^{2}\; \mathrm{d}\theta\; \mathrm{d}w\bigg)^{1/2}\\
&\quad\leq\dfrac{C}{\lambda}\left(\dfrac{1}{\sqrt{\epsilon}}\sqrt{\dfrac{\lambda}{\abs{k}}}\right)^{\frac{1}{2}}\bigg(\int_{\mathbb{R}}\int_{\mathbb{S}^{d-1}}\int_{\mathbb{S}^{d-1}}\rho_{\epsilon}(1-\theta\cdot \theta^{\prime})\; \abs{\widehat{G}(w,\bar{k},\theta^{\prime})}^{2}\; \mathrm{d}\theta^{\prime}\; \mathrm{d}\theta\;\mathrm{d}w\bigg)^{1/2}\\
&\qquad\leq \dfrac{C}{\lambda}\left(\dfrac{1}{\sqrt{\epsilon}}\sqrt{\dfrac{\lambda}{\abs{k}}}\right)^{\frac{1}{2}}\bigg(\int_{\mathbb{R}}\int_{\mathbb{S}^{d-1}}\;\abs{\widehat{G}(w,\bar{k},\theta^{\prime})}^{2}\; \mathrm{d}\theta^{\prime}\;\mathrm{d}w\bigg)^{1/2}.
\end{split}
\end{equation*}
Thus, one obtains that the $L^{2}_{t,\theta}$ norm of $T_{4}$ is estimated by
\begin{equation}\label{eq35}
\norm*{T_{4}(\cdot,k,\cdot)}_{L^{2}(\R\times\mathbb{S}^{d-1})}\leq \dfrac{C}{\lambda}\left(\dfrac{1}{\sqrt{\epsilon}}\sqrt{\dfrac{\lambda}{\abs{k}}}\right)^{\frac{1}{2}}\norm*{\widehat{G}(\cdot,\bar{k}, \cdot)}_{L^{2}(\R\times\mathbb{S}^{d-1})}.
\end{equation} 
\noindent \textit{Conclusion of the proof.} From the decomposition (\ref{eq9}) and the estimates (\ref{eq16}), (\ref{eq28}), (\ref{eq31}) and (\ref{eq35}) one concludes that 
\begin{equation}\label{eq36}
\begin{split}
&\norm*{\widehat{u}(\cdot,k,\cdot)}_{L^{2}(\R\times \mathbb{S}^{d-1})}\leq \bigg(\dfrac{1}{\sqrt{\epsilon}}\sqrt{\dfrac{\lambda}{\abs{k}}}\bigg)^{\frac{1}{2}}\bigg( \norm{\widehat{u}(\cdot,k,\cdot)}_{L^{2}(\R\times \mathbb{S}^{d-1})}+\dfrac{1}{\lambda} \norm{\mathcal{K}(\widehat{u})(\cdot,k,\cdot)}_{L^{2}(\R\times \mathbb{S}^{d-1})}\bigg)\\
&\qquad +\dfrac{C}{\lambda}\bigg(\dfrac{1}{\sqrt{\epsilon}}\sqrt{\dfrac{\lambda}{\abs{k}}}\bigg)^{\frac{1}{2}}\norm*{\widehat{u}(t_0,k,\cdot)}_{L^{2}( \mathbb{S}^{d-1})}+D\Bigg( \dfrac{1}{\lambda \epsilon}\bigg(\dfrac{1}{\sqrt{\epsilon}}\sqrt{\dfrac{\lambda}{\abs{k}}}\bigg)^{\frac{1}{q}}+\dfrac{\abs{k}}{\lambda^{2}} \bigg(\dfrac{1}{\sqrt{\epsilon}}\sqrt{\dfrac{\lambda}{\abs{k}}}\bigg)^{\frac{1}{q}}\\
& \qquad\qquad + \dfrac{1}{\lambda \epsilon}\bigg(\dfrac{1}{\sqrt{\epsilon}}\sqrt{\dfrac{\lambda}{\abs{k}}}\bigg)^{\frac{2-q}{q}}+\; \epsilon^{\frac{s}{2}}\Bigg)\norm*{(-\Delta_{v})^{s/2}\; \reallywidehat{U_{\mathcal{J}}}(\cdot,k,\cdot)}_{L^{2}(\R\times \R^{d-1})}\\ 
&\qquad\qquad\qquad +\dfrac{C}{\lambda}\left(\dfrac{1}{\sqrt{\epsilon}}\sqrt{\dfrac{\lambda}{\abs{k}}}\right)^{\frac{1}{2}}\norm*{\widehat{G}(\cdot,\bar{k}, \cdot)}_{L^{2}(\R\times \mathbb{S}^{d-1})}.
\end{split}
\end{equation}
Keep in mind that we are seeking an estimate for large frequencies in the spatial Fourier variable $k\in\R^{d}$.  Set $\epsilon=\abs{k}^{-a}$ and $\lambda=\abs{k}^{b}$ with numbers $a, b>0$ to be chosen in the sequel. Since we hope that
$$\dfrac{1}{\sqrt{\epsilon}}\sqrt{\dfrac{\lambda}{\abs{k}}}\sim \dfrac{1}{\abs{k}^{s_0}},\quad \text{for some}\quad s_0>0,$$
we can control the term
$$\dfrac{\abs{k}}{\lambda^{2}}\left(\dfrac{1}{\sqrt{\epsilon}}\sqrt{\dfrac{\lambda}{\abs{k}}}\right)^{\frac{1}{q}}\quad \text{by choosing } \quad \dfrac{\abs{k}}{\lambda^{2}}=1,$$
that is, choosing $b=1/2$.  Recalling that $q\in (1,2)$ one conclude that the leading terms are
$$\left(\dfrac{1}{\sqrt{\epsilon}}\sqrt{\dfrac{\lambda}{\abs{k}}}\right)^{\frac{1}{2}},\quad \dfrac{1}{\lambda \sqrt{\epsilon}}\left(\dfrac{1}{\sqrt{\epsilon}}\sqrt{\dfrac{\lambda}{\abs{k}}}\right)^{\frac{2-q}{q}}\quad \text{and} \quad \epsilon^{\frac{s}{2}}.$$
The best option independent of the dimension is choosing $a$ such that 
$$\max \bigg\{ \left(\dfrac{1}{\sqrt{\epsilon}}\sqrt{\dfrac{\lambda}{\abs{k}}}\right)^{\frac{1}{2}},\dfrac{1}{\lambda\sqrt{\epsilon}}\bigg\}=\epsilon^{\frac{s}{2}}.$$
In fact, 
$$\left(\dfrac{1}{\sqrt{\epsilon}}\sqrt{\dfrac{\lambda}{\abs{k}}}\right)^{\frac{1}{2}}=\left(\abs{k}^{a/2}\sqrt{\abs{k}^{-1/2}}\right)^{1/2}=\dfrac{1}{|k|^{\frac{1-2a}{8}}},$$
thus, the best option reduces to find $a$ such that
\begin{equation*}
\max \bigg\{ \left(\dfrac{1}{\sqrt{\epsilon}}\sqrt{\dfrac{\lambda}{\abs{k}}}\right)^{\frac{1}{2}},\dfrac{1}{\lambda\sqrt{\epsilon}}\bigg\} = \max \bigg\{ \dfrac{1}{|k|^{\frac{1-2a}{8}}},\dfrac{1}{|k|^{\frac{1-a}{2}}} \bigg\} = \frac{1}{|k|^{\frac{as}{2}}}.
\end{equation*}
The first term in the maximum, being the larger for $|k|\geq1$, is the constraint.  As a consequence, we must have $\frac{1-2a}{8} = \frac{as}{2}$, or, $a=\frac{1/2}{2s+1}$.
\noindent Computing from (\ref{eq36}), one concludes that for $|k|\geq1$,
\begin{equation}\label{eq37}
\begin{split}
&\norm*{\widehat{u}(\cdot,k,\cdot)}_{L^{2}(\R\times \mathbb{S}^{d-1})}\leq \dfrac{C}{\abs{k}^{\frac{as}{4}}}\bigg( \norm{\widehat{u}(\cdot,k,\cdot)}_{L^{2}(\R\times \mathbb{S}^{d-1})}+ \norm{\mathcal{K}(\widehat{u})(\cdot,k,\cdot)}_{L^{2}(\R\times \mathbb{S}^{d-1})}\\
&\quad +  \norm*{(-\Delta_{v})^{s/2}\; \reallywidehat{U_{\mathcal{J}}}(\cdot,k,\cdot)}_{L^{2}(\R\times \R^{d-1})} +\norm{\widehat{u}(t_0,k,\cdot)}_{L^{2}(\mathbb{S}^{d-1})}+\dfrac{\norm{\widehat{G}(\cdot,\bar{k},\cdot)}_{L^{2}(\R\times\mathbb{S}^{d-1})}}{ |k|^{\frac{as}{4}+\frac{1}{2}}}\bigg).
\end{split}
\end{equation}
Take $s_0=\frac{as}{4}$ and compute using Plancherel theorem
\begin{align}\label{eq38-final}
\begin{split}
&\int^{t_{1}}_{t_0}\int_{\mathbb{R}^{d}}\int_{\mathbb{S}^{d-1}}\big| (-\Delta_{x})^{s_0/2}u \big|^{2}\text{d}\theta\text{d}x\text{d}t = \int_{\R}\int_{\mathbb{R}^{d}}\int_{\mathbb{S}^{d-1}}\big| \widehat{u} \big|^{2}|k|^{2s_0}\text{d}\theta\text{d}k\text{d}w\\
&\qquad=\bigg(\int_{|k|\leq1} + \int_{|k|\geq1}\bigg)\int_{\R}\int_{\mathbb{S}^{d-1}}\big| \widehat{u} \big|^{2}|k|^{2s_0}\text{d}k\text{d}\theta\text{d}w\\
&\qquad\leq \|u\|^{2}_{L^{2}( (t_0,t_{1})\times\R^{d}_{+}\times\mathbb{S}^{d-1})} + \int_{|k|\geq1}\int_{\R}\int_{\mathbb{S}^{d-1}}\big| \widehat{u} \big|^{2}|k|^{2s_0}\text{d}\theta\text{d}w\text{d}k\,.
\end{split}
\end{align}
We apply estimate \eqref{eq37} in the second term of right noticing that denoting $A_{1}=\{|k|\geq1,|k_{d}|\leq1\}$ and $A_{2}=\{|k|\geq1,|k_{d}|\geq1\}$ it follows that
\begin{equation*}
\int_{|k|\geq1}\dfrac{\norm{\widehat{G}(\cdot,\bar{k},\cdot)}^{2}_{L^{2}(\R\times\mathbb{S}^{d-1})}}{|k|^{2s_0+1}}\mathrm{d}k=\bigg(\int_{A_{1}}+\int_{A_{2}}\bigg)\dfrac{\norm{\widehat{G}(\cdot,\bar{k},\cdot)}^{2}_{L^{2}(\R\times\mathbb{S}^{d-1})}}{|k|^{2s_0+1}}\;\mathrm{d}k\,,
\end{equation*}
where, by Plancherel theorem, one has that
\begin{align*}
\int_{A_{1}}\dfrac{\norm{\widehat{G}(\cdot,\bar{k},\cdot)}^{2}_{L^{2}(\R\times\mathbb{S}^{d-1})}}{|k|^{2s_0+1}}\;\mathrm{d}k &\leq \int_{\R^{d-1}}\norm{\widehat{G}(\cdot,\bar{k},\cdot)}^{2}_{L^{2}(\R\times\mathbb{S}^{d-1})}\;\mathrm{d}\bar{k}\;\left(\int_{\abs{k_d}\leq 1}\mathrm{d}k_d\right)\\
&\leq 2\,\|G\|^{2}_{L^{2}( (t_0,t_{1})\times\partial\R^{d}_{+}\times\mathbb{S}^{d-1})}\,.
\end{align*}
Similarly,
\begin{equation*}
\begin{split}
\int_{A_{2}}\dfrac{\norm{\widehat{G}(\cdot,\bar{k},\cdot)}^{2}_{L^{2}(\R\times\mathbb{S}^{d-1})}}{|k|^{2s_0+1}}\;\mathrm{d}k & \leq \int_{|k_{d}|\geq1}\dfrac{\norm{\widehat{G}(\cdot,\bar{k},\cdot)}^{2}_{L^{2}(\R\times\mathbb{S}^{d-1})}}{|k_{d}|^{2s_0+1}}\;\mathrm{d}k
\\
& \leq  \int_{\R^{d-1}}\norm{\widehat{G}(\cdot,\bar{k},\cdot)}^{2}_{L^{2}(\R\times\mathbb{S}^{d-1})}\;\mathrm{d}\bar{k}\;\left(\int_{\abs{k_d}\geq 1}\dfrac{\mathrm{d}k_d}{ |k_d|^{2s_0+1}}\right)
\\
&= \frac{1}{2s_0}\,\|G\|^{2}_{L^{2}( (t_0,t_{1})\times\partial\R^{d}_{+}\times\mathbb{S}^{d-1})} = \frac{1}{2s_0}\,\|\theta_{d}\,g\|^{2}_{L^{2}((t_0,t_{1})\times\Gamma^{-})}\,.
\end{split}
\end{equation*}
As a consequence, \eqref{eqthm} follows from \eqref{eq38-final} and \eqref{eq37} and the result is stablished.
\end{proof}
\begin{cor}\label{cor3.3}
Let $u$ be a solution to (\ref{eq:1}) which satisfies the conditions in Theorem \ref{thm1}. Then, for any $t_*\in (t_0,t_1)$, we have
\begin{equation}\label{eq38}
\begin{split}
\MoveEqLeft
\norm*{(-\Delta_x)^{s_0/2} \;u }_{L^{2}([t_*,t_1)\times \R^{d}\times\mathbb{S}^{d-1})}\leq C \left(\dfrac{1}{\sqrt{t_* - t_0}}+1 \right)\norm*{u}_{L^{2}((t_0,t_1)\times \R^{d}_+\times\mathbb{S}^{d-1})}
\\
&\qquad\qquad +C \left(\norm*{(-\Delta_v)^{s/2} \;U_{\mathcal{J}} }_{L^{2}((t_0,t_1)\times \R^{d}_+\times\mathbb{R}^{d-1})}+\norm*{\theta_{d}\,g}_{L^{2}((t_0,t_1)\times\Gamma^{-})} \right).
\end{split}
\end{equation}
\end{cor}
\begin{proof}
Let $\tau\in (t_0,t_1)$ be arbitrary.  Then estimate (\ref{eqthm}) gives that
\begin{equation*}
\begin{aligned}
\Vert (-\Delta_x)^{s_0/2} u \Vert_{L^{2}((\tau,t_1)\times \R^{d}\times\mathbb{S}^{d-1})} &\leq C \bigg(\Vert u(\tau) \Vert_{L^{2}( \R^{d}_+\times\mathbb{S}^{d-1})} +\Vert u \Vert_{L^{2}((t_0,t_1)\times \R^{d}_+\times\mathbb{S}^{d-1})}\\
&\hspace{-1cm} + \Vert (-\Delta_v)^{s}U_{\mathcal{J}} \Vert_{L^{2}((t_0,t_1)\times \R^{d}_+\times\mathbb{R}^{d-1})}+ \Vert \theta_{d}\,g \Vert_{L^{2}((t_0,t_1)\times\Gamma^{-})}\bigg) \,.
\end{aligned}
\end{equation*}
The result follows after taking the time average over $ \tau\in (t_0,t_*)$.
\end{proof}
\section{A priory estimates and regularity}
In this section we implement a classical program in the theory of hypo-elliptic equations that consists in proving a successive gain of regularity.  The interplay of the kinetic variable $\theta$ and the spatial variable was quantified in Theorem \ref{thm1} in previous section.  It will play an essential role here.  Of course, it is understood that solutions are assumed to be sufficiently regular so that all computations are valid.           
\subsection{Regularity estimates up to the boundary.} 
The mass equation \eqref{eq-mass} with $t'=0$ gives that
\begin{equation}\label{mass}
m(t) + \Phi_{out}(t) =  m(0) + \Phi_{in}(t),
\end{equation}
where, for $t\geq0$
\begin{align*}
m(t)&\defi \int_{\R^{d}_{+}}\int_{\mathbb{S}^{d-1}}u(t,x,\theta)\; \mathrm{d}\theta\; \mathrm{d}x\,, \qquad \text{mass},\\
\Phi_{out}(t)&\defi \int^{t}_{0}\int_{\Gamma^{+}} u\, ( \theta\cdot n(\bar{x}) )\; \mathrm{d}\theta\; \mathrm{d}\bar{x}\; \mathrm{d}\tau\,, \qquad \text{total mass output in}\, (0,t),\\
\Phi_{in}(t)&\defi \int^{t}_{0}\int_{\Gamma^{-}} g\, \big| \theta\cdot n(\bar{x}) \big|\; \mathrm{d}\theta\; \mathrm{d}\bar{x}\; \mathrm{d}\tau\,,\qquad\text{total mass input in}\, (0,t)\,.
\end{align*}
We assume in the sequel that the total input of mass is finite, $$m(0)+\sup_{t\geq0}\Phi_{in}(t)=:c(m_0,\Phi_{in})<\infty\,.$$  As a consequence, the solutions are assumed with uniform bounded mass $$\sup_{t\geq0}m(t)\leq c(m_0,\Phi_{in})\,,\qquad \sup_{t\geq0}\Phi_{out}(t)\leq c(m_0,\Phi_{in})\,.$$
The following proposition is the analog to \cite[Proposition 4.1]{alonso}, see below Proposition \ref{app-inter-x-theta}.  We refer to this reference for its proof which is based on Sobolev embedding and elementary interpolation. 
\begin{prop}\label{prop4.1}
Let $0\leq u \in L^{2}\big( (t_0,t_{1}) ; H^{s}_{\theta}\cap H^{s'}_{x}\big)$, for some $0<s'<s$, be a solution of the RTE \eqref{eq:1}.  Then, there exist an explicit $\omega>1$ and a constant $C(m_0,\Phi_{in})>0$ such that
\begin{equation*}
\begin{split}
\MoveEqLeft[0]
\int^{t_{1}}_{t_0}\norm{u}^{2\omega}_{L^{2}_{x,\theta}}\;\mathrm{d}\tau\leq C(m_0,\Phi_{in})\left( \int^{t_{1}}_{t_0}\int_{\R^{d}_{+}}\norm{u}^{2}_{H^{s}_{\theta}}\;\mathrm{d}x\;\mathrm{d}\tau+\int^{t_{1}}_{t_0}\norm{(-\Delta_x)^{s'/2}\;u}^{2}_{L^{2}_{x,\theta}}\;\mathrm{d}\tau\right).
\end{split}
\end{equation*}
\end{prop}
\begin{prop}[Full energy a priori estimate]\label{cor1}
Let $u \in L^{2}\left( (t_0,t_1)\times \R^{d}_{+}\times \mathbb{S}^{d-1}\right)$ be a solution to the RTE (\ref{eq:1}) on $(t_0,t_1)$ with $\theta_{d}\,g\in L^{2}\big((t_0,t_1) ; L^{2}_{x,\theta}\big)$. Then
\begin{equation*}
\begin{split}
\sup_{t\in (t_0,t_1)}\Big(\norm{u(t)}^{2}_{L^{2}_{x,\theta}}\Big)+\int^{t_{1}}_{t_0}\int_{\mathbb{R}^{d}_{+}}&\norm{u}^{2}_{H^{s}_{\theta}}\;\mathrm{d}x\; \mathrm{d}\tau+\int^{t_{1}}_{t_0}\norm{(-\Delta_x)^{s_0/2}u}^{2}_{L^{2}_{x,\theta}}\mathrm{d}\tau\\
&\qquad \leq C\left(\norm{u(t_0,\cdot,\cdot)}^{2}_{L^{2}_{x,\theta}}+\norm{\theta_{d}\,g}^{2}_{L^{2}((t_0,t_1)\times\Gamma^{-})}\right)\,.
\end{split}
\end{equation*}
The constant depends as $C:=C(t_1,d,s,s_0)$ with $s_0$ defined in \eqref{eqthm}.
\end{prop}
\begin{proof}
The result is a direct consequence of Theorem \ref{thm1} and energy estimates \eqref{eq2.4} and \eqref{eq2.3-1}.  The reader can find details in \cite[proof of Corollary 4.2]{alonso}.
\end{proof}
\begin{prop}\label{prop4.3} Let $0\leq u$ be a sufficiently smooth solution of the RTE (\ref{eq:1}) on $(0,T)$ and assume that 
$$G_{T}:=\sup_{0<t<T}\big\{t^{\frac{1}{w-1}}\norm{\theta_{d}\,g}^{2}_{L^{2}((t,T)\times\Gamma^{-})}\big\}<\infty\,,$$
where $\omega>1$ is given in the Proposition \ref{prop4.1}.  Then,
\begin{equation}\label{4.6a}
\norm{u(t)}^{2}_{L^{2}_{x,\theta}}\leq A\; t^{-\frac{1}{\omega-1}} + \norm{\theta_{d}\,g}^{2}_{L^{2}((t/2,T)\times\Gamma^{-})}\quad\mbox{for all}\quad 0<t<T,
\end{equation}
where
\begin{equation*}
A:=C(m_0,\Phi_{in})\max\big\{ 1 ,G_{T} \big\}\,.
\end{equation*}
Moreover, for any $0< t < T$\,,
\begin{align}\label{4.6}
\begin{split}
\int^{T}_{t}\int_{\mathbb{R}^{d}_{+}}\norm{u}^{2}_{H^{s}_{\theta}}\;\mathrm{d}x\; \mathrm{d}\tau + & \int^{T}_{t}\norm{(-\Delta_x)^{s_0/2}u}^{2}_{L^{2}_{x,\theta}}\mathrm{d}\tau \\
&\qquad \leq C\left(A\; t^{-\frac{1}{\omega-1}} + 2\norm{\theta_{d}\,g}^{2}_{L^{2}((t/2,T)\times\Gamma^{-})}\right)\,.
\end{split}
\end{align}
\end{prop}
\begin{proof}
Use Proposition \ref{prop4.1} and Proposition \ref{cor1}, with $t_{1}=T$ and $t_0=t$, to obtain
\begin{equation}\label{4.4}
\int^{T}_{t}\norm{u(\tau)}^{2\omega}_{L^{2}_{x,\theta}}\text{d}\tau\leq C\left(\norm{u(t)}^{2}_{L^{2}_{x,\theta}}+\norm{\theta_{d}\,g}^{2}_{L^{2}((t,T)\times\Gamma^{-})}\right),\quad 0<t\leq T,
\end{equation}
where $C=C(m_0,\Phi_{in})$.  Introduce
$$Y(t):=\int^{T}_{t}\norm{u(\tau)}^{2\omega}_{L^{2}_{x,\theta}}\text{d}\tau,\quad 0<t < T,$$
then, time differentiation and \eqref{4.4} lead to 
$$\dfrac{\text{d}Y}{\text{d}t}(t)+c\,Y^{\omega}(t) \leq \norm{\theta_{d}\,g}^{2w}_{L^{2}((t,T)\times\Gamma^{-})}.$$
where $c:=c(m_0,\Phi_{in},w)$.  This is a Bernoulli differential inequality.  Therefore, using a standard comparison method with explicit solutions of Bernoulli ODE, one concludes that
$$Y(t)\leq A\,t^{-\frac{1}{w-1}}\,,\qquad 0<t<T\,,$$
with $A$ as defined in the statement.
Again, using Proposition \ref{cor1} with $t_0=\tau$ and $t_{1}=T$, it follows that
\begin{equation*}
\begin{split}
\sup_{s\in (\tau,T)}&\left(\norm{u}^{2w}_{L^{2}_{x,\theta}}(s)\right) \leq C\left(\norm{u(\tau,\cdot,\cdot)}^{2w}_{L^{2}_{x,\theta}}+\norm{\theta_{d}\,g}^{2w}_{L^{2}((\tau,T)\times\Gamma^{-})}\right)\,.
\end{split}
\end{equation*}
Thus, integrating in $\tau\in(t,2t)\in(0,T)$ and using the bound on $Y(t)$, we obtain that
\begin{equation*}
t\,\norm{u(2\,t)}^{2\omega}_{L^{2}_{x,\theta}}\leq A\;t^{-\frac{1}{w-1}} +t\,\norm{\theta_{d}\,g}^{2\omega}_{L^{2}((t,T)\times\Gamma^{-})}\,,
\end{equation*}
that is,
\begin{equation*}
\norm{u(2\,t)}^{2}_{L^{2}_{x,\theta}}\leq A\;t^{-\frac{1}{w-1}} + \norm{\theta_{d}\,g}^{2}_{L^{2}((t,T)\times\Gamma^{-})}\,.
\end{equation*}
This proves estimate \eqref{4.6a} by renaming $t\rightarrow t/2$.  Estimate \eqref{4.6} follows from both, Proposition \ref{cor1} with $t_{0}=t$ and $t_{1}=T$ that give us
\begin{equation*}
\int^{T}_{t}\int_{\mathbb{R}^{d}_{+}}\norm{u}^{2}_{H^{s}_{\theta}}\;\mathrm{d}x\; \mathrm{d}\tau+\int^{T}_{t}\norm{(-\Delta_x)^{s_0/2}u}^{2}_{L^{2}_{x,\theta}}\mathrm{d}\tau \leq C\left(\norm{u(t)}^{2}_{L^{2}_{x,\theta}}+\norm{\theta_{d}\,g}^{2}_{L^{2}((t,T)\times\Gamma^{-})}\right)\,,
\end{equation*}
and estimate \eqref{4.6a} used in the right side.
\end{proof}
\begin{rem} Recall that for $0<t<T$
\begin{equation*}
\norm{(-\Delta_{v})^{s/2}U_{\mathcal{J}}}^{2}_{L^{2}\left( (t,T)\times \R^{d}_{+} \times \mathbb{R}^{d-1}\right)} \lesssim \int^{T}_{t}\int_{\mathbb{R}^{d}_{+}}\norm{u}^{2}_{H^{s}_{\theta}}\;\mathrm{d}x\; \mathrm{d}\tau\,.
\end{equation*}
Therefore, Proposition \ref{prop4.3} implies,
\begin{align*}
 \norm{(-\Delta_{v})^{s/2}U_{\mathcal{J}}}^{2}_{L^{2}\left( (t,T)\times \R^{d}_{+} \times \mathbb{R}^{d-1}\right)} \leq C\left(A\; t^{-\frac{1}{\omega-1}} + 2\norm{\theta_{d}\,g}^{2}_{L^{2}((t/2,T)\times\Gamma^{-})}\right)\,.
\end{align*}
\end{rem}
\subsubsection{Tangential spatial regularity.}  We finish this section studying higher regularity, up to the boundary, for solutions $u(t,x,\theta)$ of the RTE \eqref{eq:1}.  We first study the spatial regularity for which the analysis is relatively simple since the differential operator $\partial^{\ell}_{x}$, with $\ell=(\ell_{1},\cdots,\ell_{d})$ a multi-index with nonnegative integer entries, commutes with the RTE operator.  Indeed, if $w^{\ell}(t,x,\theta)=(\partial^{\ell}_{x}u)(t,x,\theta)$ then 
\begin{equation}\label{eq:1-x}
\partial_t w^{\ell} +\theta\cdot\nabla_x w^{\ell} =\mathcal{I}(w^{\ell}) \quad \text{in} \quad (0,T)\times \mathbb{R}^{d}_+\times \mathbb{S}^{d-1}\,.
\end{equation}
Write $\ell=(\bar{\ell},\ell_{d})$ and take $\ell_{d}=0$.  Then, $w^{(\bar{\ell},0)}(t,\bar{x},0,\theta)=g^{\bar{\ell}}(t,\bar{x},\theta)$.  As a consequence, using Proposition \ref{cor1} with $u=w^{(\bar{\ell},0)}$ it follows that
\begin{equation*}
\begin{split}
\sup_{t\in (t_0,t_1)}\Big(\norm{w^{(\bar{\ell},0)}(t)}^{2}_{L^{2}_{x,\theta}}\Big)+\int^{t_{1}}_{t_0}\int_{\mathbb{R}^{d}_{+}}&\norm{w^{(\bar{\ell},0)}}^{2}_{H^{s}_{\theta}}\;\mathrm{d}x\; \mathrm{d}\tau+\int^{t_{1}}_{t_0}\norm{(-\Delta_x)^{s_0/2}w^{(\bar{\ell},0)}}^{2}_{L^{2}_{x,\theta}}\mathrm{d}\tau\\
&\qquad \leq C\left(\norm{w^{(\bar{\ell},0)}(t_0,\cdot,\cdot)}^{2}_{L^{2}_{x,\theta}}+\norm{\theta_{d}\,g^{\bar{\ell}}}^{2}_{L^{2}((t_0,t_1)\times\Gamma^{-})}\right)\,.
\end{split}
\end{equation*}
Now, classical interpolation\footnote{$\| \partial^{\bar{\ell}}_{x}\phi\|_{L^{2}_{x}} \leq \|(-\Delta_x)^{s_0/2} \partial^{\bar{\ell}}_{x}\phi\|^{\alpha}_{L^{2}_{x}} \| \phi \|^{1-\alpha}_{L^{2}_{x}} $ for multi-index $\bar{\ell}$, $s_{0}\geq0$, and $\alpha=\frac{\bar{|\ell|}}{\bar{|\ell|}+s_0}$.} gives that
\begin{equation*}
\norm{w^{(\bar{\ell},0)}(t)}_{L^{2}_{x,\theta}} \leq C_{\ell}\norm{(-\Delta_x)^{s_0/2}w^{(\bar{\ell},0)}(t)}^{\alpha}_{L^{2}_{x,\theta}}\norm{u(t)}^{1-\alpha}_{L^{2}_{x,\theta}}\,,\qquad \alpha=\frac{\bar{|\ell|}}{\bar{|\ell|}+s_0}\,.
\end{equation*}
As a consequence, Proposition \ref{prop4.3} leads to
\begin{equation*}
\norm{w^{(\bar{\ell},0)}(t)}_{L^{2}_{x,\theta}} \leq C_{t_0}\norm{(-\Delta_x)^{s_0/2}w^{(\bar{\ell},0)}(t)}^{\alpha}_{L^{2}_{x,\theta}}\,,
\end{equation*}
for a constant depending on $C_{t_0}:=C_{\ell}(t_0,m_0,\Phi_{in},\|g\|_{L^{2}_{t,x,\theta}})$.  Therefore, 
\begin{equation*}
\begin{split}
\sup_{t\in (t_0,t_1)}\Big(\norm{w^{(\bar{\ell},0)}(t)}^{2}_{L^{2}_{x,\theta}}\Big)+\tilde{C}_{t_0}\int^{t_{1}}_{t_0}&\norm{w^{(\bar{\ell},0)}}^{2/\alpha}_{L^{2}_{x,\theta}}\\
&\leq C\left(\norm{w^{(\bar{\ell},0)}(t_0,\cdot,\cdot)}^{2}_{L^{2}_{x,\theta}}+\norm{\theta_{d}\,g^{\bar{\ell}}}^{2}_{L^{2}((t_0,t_1)\times\Gamma^{-})}\right)\,.
\end{split}
\end{equation*}
The same argument that proves Proposition \ref{prop4.3} leads to the control 
\begin{equation}\label{eq:spatial-tangential}
\norm{\partial^{(\bar{\ell},0)}_{x}u(t)}_{L^{2}_{x,\theta}}\leq C\big(t_0,m_0,\Phi_{in},\norm{\theta_{d}\,g}_{H^{\bar{\ell}}_{\bar{x}}((t_0/2,T)\times\Gamma^{-})}\big)\quad\mbox{for all}\quad 0< t_0 \leq t < T\,.
\end{equation}

\begin{prop}[Tangential spatial regularity]\label{tangential-reg} Let $0\leq u$ be a sufficiently smooth solution of the RTE (\ref{eq:1}) on $(0,T)$ and assume that 
$$\norm{\theta_{d}\,\partial^{\bar{\ell}}_{\bar{x}}g}^{2}_{L^{2}((t_0/2,T)\times\Gamma^{-})}<\infty\,,\quad 0<t_0<T\,.$$
Then,
\begin{equation}\label{4.6a}
\norm{\partial^{(\bar{\ell},0)}_{x}u(t)}^{2}_{L^{2}_{x,\theta}}\leq C\big(t_0,m_0,\Phi_{in},\norm{\theta_{d}\,g}_{H^{\bar{\ell}}_{\bar{x}}((t_0/2,T)\times\Gamma^{-})}\big)\quad\mbox{for all}\quad t_0<t<T\,.
\end{equation}
Moreover,
\begin{align}\label{4.6}
\begin{split}
\int^{T}_{t_0}\int_{\mathbb{R}^{d}_{+}}\norm{\partial^{(\bar{\ell},0)}_{x}u}^{2}_{H^{s}_{\theta}}\;\mathrm{d}x\; \mathrm{d}\tau + & \int^{T}_{t_0}\norm{(-\Delta_x)^{s_0/2}\partial^{(\bar{\ell},0)}_{x}u}^{2}_{L^{2}_{x,\theta}}\mathrm{d}\tau \\
&\qquad \leq C\big(t_0,m_0,\Phi_{in},\norm{\theta_{d}\,g}_{H^{\bar{\ell}}_{\bar{x}}((t_0/2,T)\times\Gamma^{-})}\big)\,.
\end{split}
\end{align}
\end{prop}
\begin{rem}
Note that higher regularity up to the boundary in the normal variable $x_{d}$ is not expected.  This can be seen from expression \eqref{eq8} by observing that the boundary term $\frac{\widehat{G}(w,\bar{k},\theta)}{\mathrm{i}(w + \theta \cdot k)}$ caps the decay in the Fourier variable $k_{d}\in (-\infty,\infty)$.  In addition,  $\theta$-regularity of $u$ is limited at the boundary as well due to the discontinuity on the set $\{ \theta_{d}=0 \}$.  Thus, it is unclear how much more regularity the solution enjoys with respect to the one specified in the averaging lemma. Deeper investigation of the optimal regularization near the boundary is an interesting aspect that will not be addressed in this section.  In the next section we will prove, however, that solutions are smooth in all spatial coordinates and time in the interior of the half-space.  
\end{rem}

\subsection{$\mathcal{C}^{\infty}$-- interior regularity.}
In this part of the paper we prove that solutions are smooth in all variables in the interior of the domain.  Take a smooth non negative and increasing cut-off function $\phi$ defined as
\begin{equation*}
\phi(x_{d}):=\bigg\{
\begin{array}{ccl}
0 & \text{if} & 0\leq x_{d}\leq 1\\
1 & \text{if} & x_{d}\geq2\,.
\end{array}
\end{equation*}
Set $\phi_{\epsilon}(x_{d})=\phi(x_{d}/\epsilon)$ for $\epsilon\in(0,1)$.  Multiply \eqref{eq:1} by $\phi_{\epsilon}$ to obtain
\begin{equation}\label{RTE-cutoff}
\partial_t \left( \phi_{\epsilon} u\right)+\theta \cdot \nabla_x \left( \phi_{\epsilon} u\right)=\mathcal{I}\left( \phi_{\epsilon} u\right)+\theta_{d}\,\phi^{\prime}_{\epsilon}\,u,\quad \text{in} \quad(0,\infty)\times\mathbb{R}^{d}\times\mathbb{S}^{d-1}\,,
\end{equation}
where, using an extension by zero, we interpret $\phi_{\epsilon}u$ to be defined in the whole space.  As a consequence, $\phi_{\epsilon}u$ satisfies the RTE \eqref{eq:1} with a source $F:=\theta_{d}\,\phi^{\prime}_{\epsilon}\,u$ for which we can apply the argument presented in \cite{alonso}.
\begin{prop}[Interior spatial regularity]\label{prop4b.5}
Let $u\in L^{2}\left( (t_0,t_1)\times \R^{d}_{+} \times \mathbb{S}^{d-1}\right)$ be a solution of the RTE \eqref{eq:1}.  Then, for any $l\in\mathbb{N}$ and $\epsilon\in(0,1)$ it follows that
\begin{equation}\label{4b.9}
\begin{split}
\sup_{t\in (t_0,t_1)}&\Big(\norm{(-\Delta_x)^{ls_0/2}(\phi_{\epsilon}u)(t)}^{2}_{L^{2}_{x,\theta}}\Big) + \norm{(-\Delta_x)^{(1+l)s_0/2}(\phi_{\epsilon} \,u)}^{2}_{L^{2}\left( (t_0,t_1)\times \R^{d} \times \mathbb{S}^{d-1}\right)}\\
&\qquad\qquad\leq  C(t_0,t_{1},m_0,\Phi_{in},\phi)\epsilon^{-2(1+ls_0)}\Big(1 + \norm{\theta_{d}\,g}^{2}_{L^{2}((t_0/2,t_1)\times\Gamma^{-})}\Big).
\end{split}
\end{equation}
\end{prop}
\begin{proof}
We perform induction on the regularity order $l$, proving the result for $l=s_0$ and then moving in steps $ks_0$, with $k=1,2,\cdots$.  For the basis case one notices that the arguments that led to Proposition \ref{cor1} can be applied to equation \eqref{RTE-cutoff}.  In particular, Theorem \ref{thm1} is applied in the whole space so that $g=0$ and with a source $F:=\theta_{d}\,\phi^{\prime}_{\epsilon}\,u$\footnote{The source can be treated as the term $\mathcal{K}(u)$, the bounded part of $\mathcal{I}$.}.   Then,
\begin{equation*}
\begin{split}
\sup_{t\in (t_0,t_1)}&\Big(\norm{(\phi_{\epsilon}u)(t)}^{2}_{L^{2}_{x,\theta}}\Big)+\int^{t_{1}}_{t_0}\int_{\mathbb{R}^{d}_{+}}\norm{\phi_{\epsilon}u}^{2}_{H^{s}_{\theta}}\;\mathrm{d}x\; \mathrm{d}\tau+\int^{t_{1}}_{t_0}\norm{(-\Delta_x)^{s_0/2}(\phi_{\epsilon}u)}^{2}_{L^{2}_{x,\theta}}\mathrm{d}\tau\\
&\qquad\leq C\left(\norm{(\phi_{\epsilon}u)(t_0)}^{2}_{L^{2}_{x,\theta}}+\norm{\phi_{\epsilon}u}^{2}_{L^{2}((t_0,t_1)\times\R^{d}\times\mathbb{S}^{d-1})}+\norm{F}^{2}_{L^{2}((t_0,t_1)\times\R^{d}\times\mathbb{S}^{d-1})}\right)\,.
\end{split}
\end{equation*}
The constant depends as $C:=C(d,s,s_0)$ with $s_0$ defined in \eqref{eqthm}.  In fact, since Theorem \ref{thm1} is applied to the whole space we may take $s_0=\frac{s/4}{2s+1}$.   Using Proposition \ref{prop4.3} to control the right side, we obtain that for any $\epsilon>0$
\begin{equation*}
\begin{split}
\sup_{t\in (t_0,t_1)}\Big(\norm{(\phi_{\epsilon}u)(t)}^{2}_{L^{2}_{x,\theta}}\Big)+\int^{t_{1}}_{t_0}&\norm{(-\Delta_x)^{s_0/2}(\phi_{\epsilon}u)}^{2}_{L^{2}_{x,\theta}}\mathrm{d}\tau\\
&\leq C(t_0,t_{1},m_0,\Phi_{in})\epsilon^{-2}\Big(1 + \norm{\theta_{d}\,g}^{2}_{L^{2}((t_0/2,T)\times\Gamma^{-})}\Big)\,.
\end{split}
\end{equation*}
In the last inequality we used the fact that $\| \phi'_{\epsilon} \|_{\infty}\sim \epsilon^{-1}$.  This proves the basis case.  Now assume the result to be valid for $l=k s_0$ and arbitrary $\epsilon>0$, that is
\begin{equation*}
\begin{split}
\sup_{t\in (t_0,t_1)}\Big(\norm{(-\Delta_{x})^{(k-1)s_0/2}(\phi_{\epsilon}u)(t)}^{2}_{L^{2}_{x,\theta}}\Big)&+\int^{t_{1}}_{t_0}\norm{(-\Delta_x)^{ks_0/2}(\phi_{\epsilon}u)}^{2}_{L^{2}_{x,\theta}}\mathrm{d}\tau\\
&\hspace{-1cm}\leq C(t_0,t_{1},m_0,\Phi_{in})\epsilon^{-2}\Big(1 + \norm{\theta_{d}\,g}^{2}_{L^{2}((t_0/2,T)\times\Gamma^{-})}\Big)\,,
\end{split}
\end{equation*}
and let us prove it for $l=(k+1)s_0$.  To this end differentiate \eqref{RTE-cutoff} by $(-\Delta_{x})^{ks_0/2}$ to obtain for $w^{ks_0}_{\epsilon}:=(-\Delta_{x})^{ks_0/2}(\phi_{\epsilon} u)$
\begin{equation}\label{interior-e1}
\partial_t w^{ks_0}_{\epsilon} +\theta \cdot \nabla_x w^{ks_0}_{\epsilon} =\mathcal{I}\left( w^{ks_0}_{\epsilon} \right)+\theta_{d}\,(-\Delta_{x})^{ks_0/2}(\phi^{\prime}_{\epsilon}\,u),\quad \text{in} \quad(0,\infty)\times\mathbb{R}^{d}\times\mathbb{S}^{d-1}\,.
\end{equation}
Again, we are led to
\begin{equation}\label{interior-e2}
\begin{split}
&\sup_{t\in (t_0,t_1)}\Big(\norm{w^{ks_0}_{\epsilon}(t)}^{2}_{L^{2}_{x,\theta}}\Big)+\int^{t_{1}}_{t_0}\norm{(-\Delta_x)^{s_0/2}w^{ks_0}_{\epsilon}}^{2}_{L^{2}_{x,\theta}}\mathrm{d}\tau\\
&\quad\leq C\left(\norm{w^{ks_0}_{\epsilon}(t_0)}^{2}_{L^{2}_{x,\theta}}+\norm{w^{ks_0}_{\epsilon}}^{2}_{L^{2}((t_0,t_1)\times\R^{d}\times\mathbb{S}^{d-1})}+\norm{F^{ks_0}}^{2}_{L^{2}((t_0,t_1)\times\R^{d}\times\mathbb{S}^{d-1})}\right)\,,
\end{split}
\end{equation}
where $F^{ks_0}=\theta_{d}\,(-\Delta_{x})^{ks_0/2}(\phi^{\prime}_{\epsilon}\,u)$.  In Proposition \ref{prop-app1} we prove that
\begin{equation*}
\norm{ F^{ks_0}}^{2}_{ L^{2}_{x}}   \leq \epsilon^{-2(1+ks_0)}C_{\phi}\Big(\norm{ \phi_{\epsilon/2}u}^{2}_{ L^{2}_{x} } + \norm{ (-\Delta_{x})^{ks_0/2}(\phi_{\epsilon/2}u)}^{2}_{ L^{2}_{x} } \Big)\,,
\end{equation*}
thus, we can control the latter two term in the right side of \eqref{interior-e2} with the induction hypothesis to obtain that 
\begin{equation*}
\begin{split}
\sup_{t\in (t_0,t_1)}&\Big(\norm{w^{ks_0}_{\epsilon}(t)}^{2}_{L^{2}_{x,\theta}}\Big) + \int^{t_{1}}_{t_0}\norm{(-\Delta_x)^{s_0/2}w^{ks_0}_{\epsilon}}^{2}_{L^{2}_{x,\theta}}\mathrm{d}\tau\\
&\leq C\norm{w^{ks_0}_{\epsilon}(t_0)}^{2}_{L^{2}_{x,\theta}} + \epsilon^{-2(1+ks_0)}C(t_0,t_{1},m_0,\Phi_{in},\phi)\Big(1 + \norm{\theta_{d}\,g}^{2}_{L^{2}((t_0/2,T)\times\Gamma^{-})}\Big)\,.
\end{split}
\end{equation*}
At this point one uses the argument leading to Proposition \ref{tangential-reg} to conclude that
\begin{equation*}
\begin{split}
\sup_{t\in (t_0,t_1)}\Big(\norm{w^{ks_0}_{\epsilon}(t)}^{2}_{L^{2}_{x,\theta}}\Big)& + \int^{t_{1}}_{t_0}\norm{(-\Delta_x)^{s_0/2}w^{ks_0}_{\epsilon}}^{2}_{L^{2}_{x,\theta}}\mathrm{d}\tau\\
&\leq\epsilon^{-2(1+ks_0)}C(t_0,t_{1},m_0,\Phi_{in},\phi)\Big(1 + \norm{\theta_{d}\,g}^{2}_{L^{2}((t_0/2,T)\times\Gamma^{-})}\Big)\,,
\end{split}
\end{equation*}
which proves the inductive step.
\end{proof}
In order to study the regularity in the variable $\theta\in\mathbb{S}^{d-1}$ we follow the spirit of \cite{alonso} which consists in studying the regularity in the projected variable $v\in\mathbb{R}^{d-1}$.  The main technical difficult lies in the treatments of weights that account for the geometry of the problem.  Additionally, derivation in angle is an operation that does not commutes with the equation which complicates the computations, yet, easily solved with a commutator identity of the type given in Proposition \ref{prop-app2}.

\smallskip
\noindent Take the cut-off RTE \eqref{RTE-cutoff} and apply the stereographic projection to it.  Recalling the formulas \eqref{2.4} and \eqref{2.5} it follows that $U^{\epsilon}_{\mathcal{J}}:=\frac{\phi_{\epsilon}u_{\mathcal{J}}}{\langle v \rangle^{d-1-2s}}$ satisfies 
\begin{equation}\label{4b.10}
\begin{split}
\partial_t U^{\epsilon}_{\mathcal{J}} + \theta(v) \cdot \nabla_x U^{\epsilon}_{\mathcal{J}} & = - D_0\pro{v}^{4s} (-\Delta_v)^{s}U^{\epsilon}_{\mathcal{J}} + c_{s,d}\,U^{\epsilon}_{\mathcal{J}} \\
&\qquad +\frac{\big[\mathcal{I}_{h}(\phi_{\epsilon}u)\big]_{\mathcal{J}}}{\langle v \rangle^{d-1-2s}}+\theta_{d}(v)\phi^{\prime}_{\epsilon}\,U_{\mathcal{J}}, \quad\text{in}\quad (0,T)\times\mathbb{R}^{d}\times\mathbb{R}^{d-1}\,.
\end{split}
\end{equation}
Multiply this equation by $(-\Delta_v)^{s} U^{\epsilon}_{\mathcal{J}}$ and integrate in all variables.  Using in the computation the Cauchy Schwarz inequality inequality in each term as follows
\begin{equation*}
\begin{split}
\int_{v}\Big| \theta(v) \cdot\nabla_x U^{\epsilon}_{\mathcal{J}} \times (-\Delta_v)^{s} U^{\epsilon}_{\mathcal{J}}\Big| &\leq \frac{2}{D_0}\int_{v}\Big| \theta(v) \cdot\nabla_x \frac{U^{\epsilon}_{\mathcal{J}}}{\langle v \rangle^{2s}}\Big|^{2} + \frac{D_{0}}{8}\int_{v}\Big|\langle v \rangle^{2s} (-\Delta_v)^{s} U^{\epsilon}_{\mathcal{J}}\Big|^{2}\\
&=C_{d,s}\int_{\theta}\big| \theta\cdot\nabla_x (\phi_{\epsilon}u)\big|^{2} + \frac{D_{0}}{8}\int_{v}\Big|\langle v \rangle^{2s} (-\Delta_v)^{s} U^{\epsilon}_{\mathcal{J}}\Big|^{2}\,,
\end{split}
\end{equation*}
one concludes that for $0 < t_0 < t < t_{1}$
\begin{equation*}
\begin{split}
&\big\| (-\Delta_v)^{s/2}U^{\epsilon}_{\mathcal{J}}(t) \big\|^{2}_{L^{2}_{x,v}}+\tfrac{D_0}{2} \int^{t}_{t_0} \big\| \pro{v}^{2s}(-\Delta_v)^{s}U^{\epsilon}_{\mathcal{J}} \big\|^{2}_{L^{2}_{x,v}}\text{d}\tau \leq \big\| (-\Delta_v)^{s/2}U^{\epsilon}_{\mathcal{J}}(t_0) \big\|^{2}_{L^{2}_{x,v}}\\
&\qquad\qquad+ C_{d,s}\int^{t}_{t_0}\Big(\|\phi_{\epsilon}u \|^{2}_{L^{2}_{x,\theta}} + \| \mathcal{I}_{h}(\phi_{\epsilon}u)\|^{2}_{L^{2}_{x,\theta}} + \|\theta\cdot\nabla_{x}(\phi_{\epsilon}u) \|^{2}_{L^{2}_{x,\theta}} + \|\theta_{d}\phi'_{\epsilon}u \|^{2}_{L^{2}_{x,\theta}}\Big) \text{d}\tau\,.
\end{split}
\end{equation*}
Recall that $\mathcal{I}_{h}$ is a bounded operator and given the interior spatial regularity on Proposition \ref{prop4b.5} we are led to
\begin{equation}\label{interior-theta-e1}
\begin{split}
&\big\| (-\Delta_v)^{s/2}U^{\epsilon}_{\mathcal{J}}(t) \big\|^{2}_{L^{2}_{x,v}} + \tfrac{D_0}{2} \int^{t}_{t_0} \big\| \pro{v}^{2s}(-\Delta_v)^{s}U^{\epsilon}_{\mathcal{J}} \big\|^{2}_{L^{2}_{x,v}}\text{d}\tau\\
&\hspace{1cm}\leq \big\| (-\Delta_v)^{s/2}U^{\epsilon}_{\mathcal{J}}(t_0) \big\|^{2}_{L^{2}_{x,v}} + C_{\epsilon}(t,t_0,m_0,\Phi_{in},\phi)\Big(1 + \norm{\theta_{d}\,g}^{2}_{L^{2}((t_0/2,t_1)\times\Gamma^{-})}\Big)\,.
\end{split}
\end{equation}
Now observe that
\begin{equation*}
\begin{split}
\int_{\mathbb{R}^{d-1}}\big|(-\Delta_{v})^{s/2}U^{\epsilon}_{\mathcal{J}} \big|^{2} \text{d}v &= \int_{\mathbb{R}^{d-1}} U^{\epsilon}_{\mathcal{J}}\;(-\Delta_{v})^{s}U^{\epsilon}_{\mathcal{J}}\,\text{d}v \\
&\leq \bigg(\int_{\mathbb{R}^{d-1}}\bigg| \frac{U^{\epsilon}_{\mathcal{J}}}{\langle v \rangle^{2s}}\bigg|^{2}\,\text{d}v\bigg)^{\frac{1}{2}}\bigg(\int_{\mathbb{R}^{d-1}} \big| \langle v \rangle^{2s}(-\Delta_{v})^{s}U^{\epsilon}_{\mathcal{J}} \big|^{2}\text{d}v\bigg)^{\frac{1}{2}}\,,
\end{split}
\end{equation*}	
which is equivalent to
\begin{equation*}
\| (-\Delta_{v})^{s/2}U^{\epsilon}_{\mathcal{J}} \|^{2}_{L^{2}_{v}} \leq C_{d}\| \phi_{\epsilon} u \|_{L^{2}_{\theta}}\| \langle v \rangle^{2s}(-\Delta_{v})^{s}U^{\epsilon}_{\mathcal{J}} \|_{L^{2}_{v}}\,.
\end{equation*}
This is a spherical analog to the classical Sobolev interpolation in the Euclidean plane.  As a consequence, after using Proposition \ref{prop4b.5} to control the $L^{2}_{x,\theta}$-norm of $\phi_{\epsilon}u$ we can update \eqref{interior-theta-e1} to
\begin{equation*}
\begin{split}
&\big\| (-\Delta_v)^{s/2}U^{\epsilon}_{\mathcal{J}}(t) \big\|^{2}_{L^{2}_{x,v}} + \frac{C(t_0,m_0,\Phi_{in})}{ 1 + \norm{g}^{2}_{L^{2}((t_0/2,t_1)\times\Gamma^{-})} } \int^{t}_{t_0} \big\| (-\Delta_v)^{s/2}U^{\epsilon}_{\mathcal{J}} \big\|^{4}_{L^{2}_{x,v}}\text{d}\tau\\
&\hspace{1.5cm}\leq \big\| (-\Delta_v)^{s/2}U^{\epsilon}_{\mathcal{J}}(t_0) \big\|^{2}_{L^{2}_{x,v}} + C_{\epsilon}(t,t_0,m_0,\Phi_{in},\phi)\Big(1 + \norm{\theta_{d}\,g}^{2}_{L^{2}((t_0/2,t_1)\times\Gamma^{-})}\Big)\,.
\end{split}
\end{equation*}
From here we can argue as before, for instance as in the proof of Proposition \ref{prop4.3} with $\omega=2$, to obtain that
\begin{equation}\label{interior-theta-e2}
\begin{split}
\sup_{t\in(t_0,t_{1})}\big\| (-\Delta_v)^{s/2}U^{\epsilon}_{\mathcal{J}}(t) \big\|^{2}_{L^{2}_{x,v}} &+ \int^{t_{1}}_{t_0} \big\| \pro{v}^{2s}(-\Delta_v)^{s}U^{\epsilon}_{\mathcal{J}} \big\|^{2}_{L^{2}_{x,v}}\text{d}\tau\\
&\leq C_{\epsilon}(t_{1},t_0,m_0,\Phi_{in},\phi)\Big(1 + \norm{\theta_{d}\,g}^{2}_{L^{2}((t_0/2,t_1)\times\Gamma^{-})}\Big)\,.
\end{split}
\end{equation}
Before proving interior angular regularization, we have the tools to prove interior space-time regularization as shown in the following proposition.
\begin{prop}[Space-time interior regularization]\label{prop4b.6} Let $u\in L^{2}((t_0,t_1)\times \R^{d}_{+}\times \mathbb{S}^{d-1})$ be a solution of (\ref{eq:1}). Then, for any $j_1,j_2 \in \N$ it follows that
\begin{equation}\label{space-time-regularization}
\begin{split}
\sup_{t\in(2t_0,t_{1})}\big\| (-\Delta_x)^{j_1\!\frac{s_0}{2}}\partial^{j_2}_t \big(\phi_{\epsilon}u) &\big\|^{2}_{L^{2}_{x,\theta}} + \int^{t_1}_{t_0}\big\| (-\Delta_x)^{(1 + j_{1})\frac{s_0}{2}}\partial^{j_2}_t \big(\phi_{\epsilon}u) \big\|^{2}_{L^{2}_{x,\theta}}\\
&\leq C_{\epsilon}(t_{1},t_0,m_0,\Phi_{in},\phi)\Big(1 + \norm{\theta_{d}\,g}^{2}_{L^{2}((t_0/2,t_1)\times\Gamma^{-})}\Big)\,,
\end{split}
\end{equation}
for any $0<2t_0\leq t_{1}$.
\end{prop}
\begin{proof}
Recall that
\begin{equation*}
\partial_t \left( \phi_{\epsilon} u\right)=\mathcal{I}\left( \phi_{\epsilon} u\right)+\theta_{d}\,\phi^{\prime}_{\epsilon}\,u - \theta \cdot \nabla_x \left( \phi_{\epsilon} u\right),\quad \text{in} \quad(0,\infty)\times\mathbb{R}^{d}\times\mathbb{S}^{d-1}\,.
\end{equation*}
Therefore,
\begin{equation*}
\int^{t_{1}}_{t_0}\big\| \partial_t \left( \phi_{\epsilon} u\right) \big\|^{2}_{L^{2}_{x,\theta}}\leq 2\int^{t_{1}}_{t_0}\big\| \mathcal{I}\left( \phi_{\epsilon} u\right)\big\|^{2}_{L^{2}_{x,\theta}} + 2\int^{t_{1}}_{t_0}\big\| \theta_{d}\,\phi^{\prime}_{\epsilon}\,u - \theta \cdot \nabla_x \left( \phi_{\epsilon} u\right) \big\|^{2}_{L^{2}_{x,\theta}}\,.
\end{equation*}	
Using Proposition \ref{prop4b.5} one controls the second term in the right side
\begin{equation*}
\int^{t_{1}}_{t_0}\big\| \theta_{d}\,\phi^{\prime}_{\epsilon}\,u - \theta \cdot \nabla_x \left( \phi_{\epsilon} u\right) \big\|^{2}_{L^{2}_{x,\theta}}\text{d}\tau \leq C_{\epsilon}\Big(1 + \norm{\theta_{d}\,g}^{2}_{L^{2}((t_0/2,t_1)\times\Gamma^{-})}\Big)\,.
\end{equation*} 	
In addition, thanks to formulas \eqref{2.4} and \eqref{2.5}
\begin{equation*}
\big\| \mathcal{I}(\phi_{\epsilon} u ) \|^{2}_{L^{2}_{\theta}} \sim b(1)\big\| \langle v \rangle^{2s}(-\Delta_{v})^{s}U^{\epsilon}_{\mathcal{J}} \big\|^{2}_{L^{2}_{v}} +  \|\phi_{\epsilon} u\|^{2}_{L^{2}_{v}}\,.
\end{equation*}
Thus, estimate \eqref{interior-theta-e2} leads to the bound
\begin{equation*}
\int^{t_{1}}_{t_0}\big\| \mathcal{I}(\phi_{\epsilon} u ) \|^{2}_{L^{2}_{\theta}}\text{d}\tau \leq C_{\epsilon}\Big(1 + \norm{\theta_{d}\,g}^{2}_{L^{2}((t_0/2,t_1)\times\Gamma^{-})}\Big)\,,
\end{equation*}	
which in turn leads to
\begin{equation}\label{tempe1}
\int^{t_{1}}_{t_0}\big\| \partial_t \left( \phi_{\epsilon} u\right) \big\|^{2}_{L^{2}_{x,\theta}}\leq C_{\epsilon}\Big(1 + \norm{\theta_{d}\,g}^{2}_{L^{2}((t_0/2,t_1)\times\Gamma^{-})}\Big)\,,\quad \epsilon>0\,.
\end{equation}
Let us prove that $y^{j}_{\epsilon}:=\partial^{j}_{t}(\phi_{\epsilon} u)$, with $j=1,2,\cdots$, is smooth in the spatial variable.  Time differentiation commutes with the RTE, as a consequence, note that 
\begin{equation}\label{RTE-time-j}
\partial_t y^{j}_{\epsilon} +\theta \cdot \nabla_x y^{j}_{\epsilon}  = \mathcal{I}\left(  y^{j}_{\epsilon} \right)+\theta_{d}\,\phi^{\prime}_{\epsilon}\,y^{j}_{\epsilon/2}\,.
\end{equation}
Then, we can invoke Theorem \ref{thm1} applied in the whole space, with zero boundary $g=0$, and with a source $F^{j}_{\epsilon}:=\theta_{d}\,\phi^{\prime}_{\epsilon}y^{j}_{\epsilon/2}$.  As a consequence,
\begin{equation}\label{EE-temporal}
\begin{split}
\sup_{t\in (t_0,t_1)}\Big(\norm{y^{j}_{\epsilon}}^{2}_{L^{2}_{x,\theta}}\Big)&+\int^{t_{1}}_{t_0}\norm{(-\Delta_x)^{s_0/2}y^{j}_{\epsilon}}^{2}_{L^{2}_{x,\theta}}\mathrm{d}\tau\leq C\Big( \norm{y^{j}_{\epsilon}(t_0)}^{2}_{L^{2}_{x,\theta}}\\
&\qquad\qquad +\norm{y^{j}_{\epsilon}}^{2}_{L^{2}((t_0,t_1)\times\R^{d}\times\mathbb{S}^{d-1})}+\norm{F^{j}_{\epsilon}}^{2}_{L^{2}((t_0,t_1)\times\R^{d}\times\mathbb{S}^{d-1})}\Big)\,,
\end{split}
\end{equation}
where the constant depends as $C:=C(d,s,s_0)$ with $s_0$ defined in \eqref{eqthm}. 
 
\smallskip
\noindent  Since \eqref{tempe1} is valid for any $\epsilon>0$, one concludes that for $j=1$ it holds that
\begin{equation*}
\norm{y^{1}_{\epsilon}}^{2}_{L^{2}((t_0,t_1)\times\R^{d}\times\mathbb{S}^{d-1})}+\norm{F^{1}_{\epsilon}}^{2}_{L^{2}((t_0,t_1)\times\R^{d}\times\mathbb{S}^{d-1})}\leq C_{\epsilon}\Big(1 + \norm{\theta_{d}\,g}^{2}_{L^{2}((t_0/2,t_1)\times\Gamma^{-})}\Big)\,,
\end{equation*}
and consequently,
\begin{equation*}
\begin{split}
\sup_{t\in (t_0,t_1)}\Big(\norm{y^{1}_{\epsilon}}^{2}_{L^{2}_{x,\theta}}\Big)&+\int^{t_{1}}_{t_0}\norm{(-\Delta_x)^{s_0/2}y^{1}_{\epsilon}}^{2}_{L^{2}_{x,\theta}}\mathrm{d}\tau\\
&\qquad \leq C\norm{y^{1}_{\epsilon}(t_0)}^{2}_{L^{2}_{x,\theta}}  + C_{\epsilon}\Big(1 + \norm{\theta_{d}\,g}^{2}_{L^{2}((t_0/2,t_1)\times\Gamma^{-})}\Big)\,,
\end{split}
\end{equation*}
from which the following estimate is readily obtained for, say, $0<\frac{3}{2}t_{0}\leq t_{1}$
\begin{equation}\label{temporale2}
\begin{split}
\sup_{t\in (\frac{3}{2}t_0,t_1)}\Big(\norm{y^{1}_{\epsilon}}^{2}_{L^{2}_{x,\theta}}\Big) +  \int^{t_{1}}_{t_0}\norm{(-\Delta_x)^{s_0/2}y^{1}_{\epsilon}}^{2}_{L^{2}_{x,\theta}}\mathrm{d}\tau \leq C_{\epsilon}\Big(1 + \norm{\theta_{d}\,g}^{2}_{L^{2}((t_0/2,t_1)\times\Gamma^{-})}\Big)\,.
\end{split}
\end{equation} 
Estimate \eqref{temporale2} works as the base case for an induction argument along the lines of the proof of Proposition \ref{prop4b.5} which proves estimate \eqref{space-time-regularization} for $j_{1}\in\mathbb{N}$ and $j_{2}=1$.

\smallskip
\noindent
For the general case $j_{2}\geq1$ argue again by induction, using estimate \eqref{space-time-regularization} for $j_{1}\in\mathbb{N}$ and $j_{2}=1$ as the base case and with induction hypothesis given by the same estimate for $j_{1}\in\mathbb{N}$ and $j_{2}=j$.  In order to prove \eqref{space-time-regularization} for $j+1$, note that thanks to the induction hypothesis $\nabla_{x}y^{j}_{\epsilon}$ and $y^{j}_{\epsilon/2}$ belong to $L^{2}_{t,x,\theta}$ with the right control with respect to the boundary $g$.  Additionally, all $y^{j}_{\epsilon}$ satisfy the same equation \eqref{RTE-time-j}, therefore, we can redo the steps that proved \eqref{interior-theta-e2} to obtain that
\begin{equation*}
\begin{split}
\sup_{t\in(t_0,t_{1})}\big\| (-\Delta_v)^{s/2}Y^{j,\epsilon}_{\mathcal{J}}(t) \big\|^{2}_{L^{2}_{x,v}} &+ \int^{t_{1}}_{t_0} \big\| \pro{v}^{2s}(-\Delta_v)^{s}Y^{j,\epsilon}_{\mathcal{J}} \big\|^{2}_{L^{2}_{x,v}}\text{d}\tau\\
&\leq C_{\epsilon}(t_{1},t_0,m_0,\Phi_{in},\phi)\Big(1 + \norm{\theta_{d}\,g}^{2}_{L^{2}((t_0/2,t_1)\times\Gamma^{-})}\Big)\,,
\end{split}
\end{equation*}
where $Y^{j,\epsilon}_{\mathcal{J}} = \frac{\big[ y^{j}_{\epsilon}\big]_{\mathcal{J}}}{\langle v \rangle^{d-1-2s}}$.  This implies thanks to the argument leading to \eqref{tempe1} that
\begin{equation*}
\int^{t_{1}}_{t_0}\big\| y^{j+1}_{\epsilon} \big\|^{2}_{L^{2}_{x,\theta}}\leq C_{\epsilon}\Big(1 + \norm{\theta_{d}\,g}^{2}_{L^{2}((t_0/2,t_1)\times\Gamma^{-})}\Big)\,,\quad \epsilon>0\,.
\end{equation*}
Invoking again Theorem \ref{thm1} we can use the energy estimate \eqref{EE-temporal} with $j+1$ in place of $j$.  From this point we conclude as done for $y^{1}_{\epsilon}$.
\end{proof}
\begin{prop}[Interior angular regularity]\label{prop4b.6}
Let $u\in L^{2}\left( (t_0,t_1)\times \R^{d}_{+} \times \mathbb{S}^{d-1}\right)$ be a solution of the RTE \eqref{eq:1}.  Assume in addition that the scattering kernel \eqref{b-scattering} satisfies for some integer $N_0\geq 1$
\begin{equation}\label{4b.8}
h(z)=\dfrac{\tilde{b}(z)}{(1-z)^{1 + s}} \in \mathcal{C}^{N_0}\big( [-1,1] \big).
\end{equation}
Then, for any $j_{1},\,j_{2}\in\mathbb{N}$ such that $0 \leq \frac{(j_{2} - 1) s}{2} \leq \lfloor N_0 \rfloor$, and $\epsilon\in(0,1)$ it follows that
\begin{equation}\label{4b.9}
\begin{split}
\sup_{t\in(t_0,t_{1})}\norm{(-\Delta_{x})^{\frac{j_{1}}{2}}(-\Delta_v)^{j_{2}\frac{s}{2}}\,&U^{\epsilon}_{\mathcal{J}}}^{2}_{L^{2}_{x,v}}  + \int^{t_{1}}_{t_0} \big\| \pro{v}^{2s}(-\Delta_{x})^{\frac{j_{1}}{2}}(-\Delta_v)^{(1+ j_{2})\frac{s}{2}}\,U^{\epsilon}_{\mathcal{J}} \big\|^{2}_{L^{2}_{x,v}}\text{d}\tau\\
&\!\!\! \leq  C_{\epsilon,j_1,j_2}(t_0,t_{1},m_0,\Phi_{in},\phi)\Big(1 + \norm{\theta_{d}\,g}^{2}_{L^{2}((t_0/2,t_1)\times\Gamma^{-})}\Big)\,,
\end{split}
\end{equation}	
where we recall that $U^{\epsilon}_{\mathcal{J}}=\frac{\phi_{\epsilon}\,u_{\mathcal{J}}}{\pro{v}^{d-1-2s}}$.	
\end{prop}
\begin{proof}
The case $j_{1}\in\mathbb{N}$ and $j_{2}=1$ is clear from the argument leading to estimate \eqref{interior-theta-e2} after spatial differentiation of the equation given the smoothness of the spatial variable Proposition \ref{prop4b.5}.  For the general case assume, as inductive hypothesis, that \eqref{4b.9} for $j_{1}\in\mathbb{N}$ and $j_{2}=j$.  Apply the operator $(-\Delta_{x})^{\frac{j_{1}}{2}}\partial^{\kappa}_{v}(1-\Delta_{v})^{\frac{\alpha}{2}}$ to the projected RTE \eqref{4b.10} with $\kappa$ multi-index and $\alpha\in(0,1)$ such that $|\kappa| + \alpha = j\frac{s}{2}$.  Then for $V^{j,\epsilon}_{\mathcal{J}} = (-\Delta_{x})^{\frac{j_{1}}{2}}\partial^{\kappa}_{v}(1-\Delta_{v})^{\frac{\alpha}{2}}U^{\epsilon}_{\mathcal{J}}$ it holds that 
\begin{equation}\label{equation-angular-regularization}
\begin{split}
\partial_t V^{j,\epsilon}_{\mathcal{J}}  = - D_0\pro{v}^{4s} (-\Delta_v)^{s}V^{j,\epsilon}_{\mathcal{J}} + c_{s,d}\,V^{j,\epsilon}_{\mathcal{J}} + F^{j,\epsilon}, \quad\text{in}\quad (0,T)\times\mathbb{R}^{d}\times\mathbb{R}^{d-1}\,,
\end{split}
\end{equation}
where $F^{j,\epsilon} = \sum^{4}_{i=1}F^{j,\epsilon}_{i}$.  Let us estimate each of these sources as we define them starting with the angular fractional Laplacian commutation term
\begin{align*}
F^{j,\epsilon}_{1} :=\sum_{|m|\leq |\kappa|-1}\bigg(\begin{array}{c}\kappa\\m\end{array}\bigg)\Big( \big(\partial^{\kappa-m}_{v}\langle v \rangle^{4s}\big)&\times(-\Delta_{v})^{s}(-\Delta_{x})^{\frac{j_{1}}{2}}\partial^{m}_{v}(1-\Delta_{v})^{\frac{\alpha}{2}}U^{\epsilon}_{\mathcal{J}} \\
&+ \mathcal{R}^{1}_{m}\big( (-\Delta_{v})^{s}(-\Delta_{x})^{\frac{j_{1}}{2}}\partial^{m}_{v}U^{\epsilon}_{\mathcal{J}} \big) \Big)\,. 
\end{align*}
In this identity we used first the product rule of classical differentiation to handle the operator $\partial^{\kappa}_{v}$ and, then, Proposition \ref{prop-app2} to commute the fractional differentiation.  The operator $\mathcal{R}^{1}_{m}$ satisfies for all $|m|+\alpha\leq |\kappa|-1+\alpha = j\frac{s}{2}-1$
\begin{align*}
\big\| \langle \cdot \rangle^{-2s}\mathcal{R}^{1}_{m}&\big( (-\Delta_{v})^{s}(-\Delta_{x})^{\frac{j_{1}}{2}}\partial^{m}_{v}U^{\epsilon}_{\mathcal{J}} \big) \big\|_{L^{2}_{v}} \leq C \big\| \langle \cdot \rangle^{2s} (-\Delta_{v})^{s}(-\Delta_{x})^{\frac{j_{1}}{2}}\partial^{m}_{v}U^{\epsilon}_{\mathcal{J}}\big\|_{L^{2}_{v}}\\
& \leq C \big\| \langle \cdot \rangle^{2s} (-\Delta_{v})^{s}(-\Delta_{x})^{\frac{j_{1}}{2}}U^{\epsilon}_{\mathcal{J}}\big\|^{1-\frac{|m|+\alpha}{js/2}}_{L^{2}_{v}}\big\| \langle \cdot \rangle^{2s} (-\Delta_{v})^{s}V^{j,\epsilon}_{\mathcal{J}}\big\|^{\frac{|m|+\alpha}{js/2}}_{L^{2}_{v}}\,.
\end{align*}  
The last inequality follows by interpolation.  Furthermore, invoking interpolation again
\begin{align*}
&\big\| \langle \cdot \rangle^{-2s} \big(\partial^{\kappa-m}_{v}\langle v \rangle^{4s}\big)\times(-\Delta_{v})^{s}(-\Delta_{x})^{\frac{j_{1}}{2}}\partial^{m}_{v}(1-\Delta_{v})^{\frac{\alpha}{2}}U^{\epsilon}_{\mathcal{J}} \big\|_{L^{2}_{v}} \\
&\hspace{2cm}\leq C \big\| \langle \cdot \rangle^{2s} (-\Delta_{v})^{s}(-\Delta_{x})^{\frac{j_{1}}{2}}U^{\epsilon}_{\mathcal{J}}\big\|^{1-\frac{|m|+\alpha}{js/2}}_{L^{2}_{v}}\big\| \langle \cdot \rangle^{2s} (-\Delta_{v})^{s}V^{j,\epsilon}_{\mathcal{J}}\big\|^{\frac{|m|+\alpha}{js/2}}_{L^{2}_{v}}\\
&\hspace{4cm} +C \big\| \langle \cdot \rangle^{2s} (-\Delta_{v})^{s}(-\Delta_{x})^{\frac{j_{1}}{2}}U^{\epsilon}_{\mathcal{J}}\big\|_{L^{2}_{v}}\,.
\end{align*}
Similarly, the spatial gradient term is given by
\begin{align*}
&F^{j,\epsilon}_{2} := (-\Delta_{x})^{\frac{j_{1}}{2}}\partial^{\kappa}_{v}\big(1-\Delta_{v}\big)^{\frac{\alpha}{2}}\Big(\theta(v)\cdot\nabla_{x}U^{\epsilon}_{\mathcal{J}}\Big)\\
&=\sum_{|m|\leq |\kappa|}\bigg(\begin{array}{c}\kappa\\m\end{array}\bigg)\Big( \partial^{\kappa-m}_{v}\theta(v)\cdot\nabla_{x}(-\Delta_{x})^{\frac{j_{1}}{2}}\partial^{m}_{v}(1-\Delta_{v})^{\frac{\alpha}{2}}U^{\epsilon}_{\mathcal{J}} + \mathcal{R}^{2}_{m}\big( \partial_{x_{i}}(-\Delta_{x})^{\frac{j_{1}}{2}}\partial^{m}_{v}U^{\epsilon}_{\mathcal{J}} \big) \Big)\,, 
\end{align*}
where Proposition \ref{prop-app2} was used for the commutation.  Given the explicit form of $\theta_{i}(v)=\mathcal{J}_{i}(v)$ presented in \eqref{stero2} it follows that
\begin{align*}
\big\| \langle \cdot \rangle^{-2s}\mathcal{R}^{2}_{m}\big( \partial_{x_i}(-\Delta_{x})^{\frac{j_{1}}{2}}&\partial^{m}_{v}U^{\epsilon}_{\mathcal{J}} \big) \big\|^{2}_{L^{2}_{v}}\leq C\|\langle \cdot \rangle^{-2s} \nabla_{x}(-\Delta_{x})^{\frac{j_{1}}{2}}\partial^{m}_{v}U^{\epsilon}_{\mathcal{J}} \|^{2}_{L^{2}_{v}}\\
&\qquad \leq C\,\|\langle \cdot \rangle^{-2s} \nabla_{x}(-\Delta_{x})^{\frac{j_{1}}{2}}U^{\epsilon}_{\mathcal{J}} \|^{2}_{L^{2}_{v}} + C\| \langle \cdot \rangle^{-2s} \nabla_{x}V^{j,\epsilon}_{\mathcal{J}} \|^{2}_{L^{2}_{v}}\,.
\end{align*}
In the same way,
\begin{align*}
\big\| \langle\cdot \rangle^{-2s} \partial^{\kappa-m}_{v}\theta(v)\cdot\nabla_{x}(-\Delta_{x})^{\frac{j_{1}}{2}}&\partial^{m}_{v}(1-\Delta_{v})^{\frac{\alpha}{2}}U^{\epsilon}_{\mathcal{J}} \big\|^{2}_{L^{2}_{v}} \\
&\leq C\,\|\langle \cdot \rangle^{-2s} \nabla_{x}(-\Delta_{x})^{\frac{j_{1}}{2}}U^{\epsilon}_{\mathcal{J}} \|^{2}_{L^{2}_{v}} + C\| \langle \cdot \rangle^{-2s} \nabla_{x}V^{j,\epsilon}_{\mathcal{J}} \|^{2}_{L^{2}_{v}}\,.
\end{align*}
Repeating previous argument and using Proposition \ref{prop-app1}, the following estimate for $$F^{j,\epsilon}_{3}:=(-\Delta_{x})^{\frac{j_{1}}{2}}\partial^{\kappa}_{v}\big(1-\Delta_{v}\big)^{\frac{\alpha}{2}}\big(\theta_{d}\,\phi'_{\epsilon}\,U_{\mathcal{J}}\big)$$ holds
\begin{equation*}
\| \langle \cdot \rangle^{-2s}F^{j,\epsilon}_{3} \|^{2}_{L^{2}_{v}} \leq C_{\epsilon}\,\|\langle \cdot \rangle^{-2s} (-\Delta_{x})^{\frac{j_{1}}{2}}U^{j,\epsilon/2}_{\mathcal{J}} \|^{2}_{L^{2}_{v}} + C_{\epsilon}\| \langle \cdot \rangle^{-2s}V^{j,\epsilon/2}_{\mathcal{J}} \|^{2}_{L^{2}_{v}}\,.
\end{equation*}
Finally, for $$F^{j,\epsilon}_{4} = (-\Delta_{x})^{\frac{j_{1}}{2}}\partial^{\kappa}_{v}\big(1-\Delta_{v}\big)^{\frac{\alpha}{2}}\bigg(\frac{\big[\mathcal{I}_{h}(\phi_{\epsilon}u)\big]_{\mathcal{J}}}{\langle v \rangle^{d-1-2s}} \bigg)\,,$$
note that
\begin{equation*}
h\big( \theta\cdot\theta' \big) = h\Big(1 - 2\,\dfrac{\abs{v-v^{\prime}}^{2}}{\pro{v}^{2}\pro{v^{\prime}}^{2}}\Big)\,.
\end{equation*}
Therefore,
\begin{equation*}
F^{j,\epsilon}_{4} \sim \frac{\big[\mathcal{I}_{\tilde{h}}\big( (-\Delta_{x})^{\frac{j_{1}}{2}}\phi_{\epsilon}u \big)\big]_{\mathcal{J}}}{\langle v \rangle^{d-1-2s}}\,,\quad\text{where}\quad \tilde{h} = \big(1-\Delta_{v}\big)^{\frac{js/2}{2}}\bigg( h\Big(1 - 2\,\dfrac{\abs{v-v^{\prime}}^{2}}{\pro{v}^{2}\pro{v^{\prime}}^{2}}\Big) \bigg)\,.
\end{equation*}
As a consequence, since $j\frac{s}{2}\leq N_{0}$, the estimate holds
\begin{equation*}
\| \langle \cdot \rangle^{-2s}F^{j,\epsilon}_{4} \|^{2}_{L^{2}_{v}} \leq C(\| h \|_{\mathcal{C}^{N_0}})\big\| (-\Delta_{x})^{\frac{j_{1}}{2}}u^{\epsilon}\big\|^{2}_{L^{2}_{\theta}} \sim \big\|\langle \cdot \rangle^{-2s} (-\Delta_{x})^{\frac{j_{1}}{2}}U^{\epsilon}_{\mathcal{J}}\big\|^{2}_{L^{2}_{v}}\,.
\end{equation*}
Overall, these estimates together with the induction hypothesis lead to the following control valid for any $\delta>0$
\begin{equation}\label{estimate-angular-source}
\begin{split}
\int^{t_{1}}_{t_0}\big\| \langle \cdot \rangle^{-2s} F^{j,\epsilon} \big\|^{2}_{L^{2}_{x,v}}\text{d}\tau&\leq C_{\delta}\Big(1 + \norm{\theta_{d}\,g}^{2}_{L^{2}((t_0/2,t_1)\times\Gamma^{-})}\Big) \\
&\qquad\qquad + \delta \int^{t_{1}}_{t_0}\big\| \langle \cdot \rangle^{2s} (-\Delta_{v})^{s}V^{j,\epsilon}_{\mathcal{J}}\big\|^{2}_{L^{2}_{x,v}}\text{d}\tau\,.
\end{split}
\end{equation}
The proof follows from estimate \eqref{estimate-angular-source} multiplying equation \eqref{equation-angular-regularization} by $(-\Delta_{v})^{s}V^{j,\epsilon}$ integra\-ting in all variables and choosing $\delta=\frac{D_0}{2}$. 
\end{proof}
\section{Existence and uniqueness of solution}\label{sec5}
In this section we prove the well-posedness for the RTE \eqref{eq:1} using the Lumer--Phillips Theorem.  The main step consists in describing the domain of the operator
\begin{equation*}
\mathcal{L}(u) := \theta\cdot\nabla_{x}u - \mathcal{I}(u)\,.
\end{equation*} 
The central issue is to guarantee the existence of a trace mapping in such domain for which the Green's formula is valid.  To this end, consider the Banach space
\begin{equation*}
H^{s}_{\mathcal{L}} = \big\{ u \in L^{2}_{x}\cap H^{s}_{\theta}\big(\mathbb{R}^{d}_{+}\times\mathbb{S}^{d-1}\big) \, \big| \, \mathcal{L}(u) \in L^{2}\big(\mathbb{R}^{d}_{+}\times\mathbb{S}^{d-1}\big)  \big\}\,,
\end{equation*}
with norm
\begin{equation*}
\|u\|^{2}_{H^{s}_{\mathcal{L}}} : = \int_{\mathbb{R}^{d-1}_{+}}\|u\|^{2}_{H^{s}_{\theta}}\,{\rm d}x + \| \mathcal{L}u \|^{2}_{L^{2}_{x,\theta}} \,.
\end{equation*}
\begin{prop}[Local traces $\gamma^{\pm}$]\label{Traceprop}
Let $K^{\pm}$ be a compact set of $$\Gamma^{\pm}=\big\{ (x,\theta)\in \big\{x \,\big|\, x_{d}=0 \big\}\times\mathbb{S}^{d-1}\,\big| \, \pm\theta_{d} < 0 \big\}\,.$$
Then, for any $u\in H^{1}_{x}\cap H^{s}_{\mathcal{L}}$ the trace mappings
\begin{equation*}
u\rightarrow \gamma_{K^{\pm}}(u) := u\big|_{K^{\pm}}
\end{equation*}  
satisfy the bound
\begin{equation}\label{contrace}
\|\gamma_{K^{\pm}}(u)\|_{L^{2}(K^{\pm})} \leq C_{K^{\pm}}\| u \|_{H^{s}_{\mathcal{L}}}\,.
\end{equation}
\end{prop}
\begin{proof}
Fix $K^{+}\subset\Gamma^{+}$ compact and let $\varphi_{K^{+}}(\theta)\geq0$ be a smooth function equal to unity in $K^{+}$ and compactly supported in $\{\theta_{d}<0\}$.  For any $w\in H^{1}_{x}\cap H^{s}_{\mathcal{L}}$ it follows that $\mathcal{L}(w) =: f \in L^{2}_{x,\theta}$ and, since the Green's formula is valid in $H^{1}_{x}$, 
\begin{align*}
\int_{\mathbb{R}^{d}_{+}}\int_{\mathbb{S}^{d-1}} f\,\varphi^{2}_{K^{+}}& w \,{\rm d}\theta \,{\rm d}x = \int_{\mathbb{R}^{d}_{+}}\int_{\mathbb{S}^{d-1}} \mathcal{L}(w)\,\varphi^{2}_{K^{+}}w \,{\rm d}\theta \,{\rm d}x \\
&= \int_{\Gamma^{+}} (\varphi_{K^{+}}w)^{2}\,\theta_{d}\,{\rm d}\theta \,{\rm d}\bar{x} - \int_{\mathbb{R}^{d}_{+}}\int_{\mathbb{S}^{d-1}} \mathcal{I}(w)\,\varphi^{2}_{K^{+}}w \,{\rm d}\theta \,{\rm d}x\,.
\end{align*}
Recalling the weak formulation \eqref{1.6}, it follows that
\begin{align*}
\Big|\int_{\mathbb{R}^{d}_{+}}\int_{\mathbb{S}^{d-1}} \mathcal{I}(w)\,\varphi^{2}_{K^{+}}w \,{\rm d}\theta \,{\rm d}x \Big| &\leq \int_{\mathbb{R}^{d-1}_{+}} \|w\|_{H^{s}_{\theta}}\| \varphi^{2}_{K^{+}}w\|_{H^{s}_{\theta}}\text{d}x \\
&\leq  C_{K^{+}}\int_{\mathbb{R}^{d-1}_{+}} \|w\|^{2}_{H^{s}_{\theta}}\text{d}x\,.
\end{align*}
Also,
\begin{align*}
\Big| \int_{\mathbb{R}^{d}_{+}}\int_{\mathbb{S}^{d-1}} f\,\varphi^{2}_{K^{+}} w \,{\rm d}\theta \,{\rm d}x \Big| &\leq \int_{ \mathbb{R}^{d} }\|f\|_{H^{-s}_{\theta}}\|w\,\varphi^{2}_{K^{+}}\|_{H^{s}_{\theta}}{\rm d}x \\
&\leq C_{K^{+}}\Big( \int_{\mathbb{R}^{d}_{+}}\|f \|^{2}_{H^{-s}_{\theta}} {\rm d}x+ \int_{\mathbb{R}^{d-1}}\|w\|^{2}_{H^{s}_{\theta}}{\rm d}x \Big)\,.
\end{align*}
Consequently, we obtain that
\begin{align*}
{\rm dist}\big(\{\theta_d=0\},K^{+}\big)\int_{K^{+}} | w |^{2}\,{\rm d}\theta \,{\rm d}\bar{x} &\leq \int_{K^{+}} | w |^{2}\,|\theta_{d}|\,{\rm d}\theta \,{\rm d}\bar{x}\\
&\leq \int_{\Gamma^{+}} (\varphi_{K^{+}}w)^{2}\,|\theta_{d}|\,{\rm d}\theta \,{\rm d}\bar{x}\leq C_{K^{+}}\| w \|^{2}_{H^{s}_{\mathcal{L}}}\,.
\end{align*}
This proves the estimate for $\gamma_{K^{+}}$.  For the trace mapping $\gamma_{K^{-}}$ the proof is similar.
\end{proof}
\noindent
Thanks to Proposition \ref{Traceprop}, the trace mappings $\gamma^{\pm}(\cdot)$ can be extended by continuity from $H^{1}_{x}\cap H^{s}_{\mathcal{L}}$ to its closure $\overline{ H^{1}_{x}\cap H^{s}_{\mathcal{L}} }$ in $H^{s}_{\mathcal{L}}$.  In other words, functions in $\overline{ H^{1}_{x}\cap H^{s}_{\mathcal{L}} }$ have well defined traces $\gamma^{\pm}$ in $L^{2}_{loc}(\Gamma^{\pm})$. 
\begin{thm}\label{Tracetheorem}
Let $H^{s}_{\Gamma_{loc}} = \big\{ u\in H^{s}_{\mathcal{L}} \, \big|\, \gamma^{\pm}(u)\in L^{2}_{loc}(\Gamma^{\pm}) \big\}$\footnote{In the defintion of $H^{s}_{\Gamma_{loc}}$ is implicit that $\gamma^{\pm}(\cdot)$ are well defined and agree with $L^{2}_{loc}(\Gamma^{\pm})$ functions.}.  Then
\begin{equation*}
\overline{H^{1}_{x}\cap H^{s}_{\mathcal{L}}} = H^{s}_{\Gamma_{loc}}\,.
\end{equation*}
\end{thm}
\begin{proof}
We know, due to Proposition \ref{Traceprop}, that $\overline{H^{1}_{x}\cap H^{s}_{\mathcal{L}}} \subset H^{s}_{\Gamma_{loc}}$.  Let us prove the opposite inclusion.  The strategy consists in proving that the space $H^{1}_{x} \cap H^{s}_{\mathcal{L}}\cap\{\gamma^{-}(\cdot)=0\}$ is dense in $H^{s}_{\mathcal{L}}\cap\{\gamma^{-}(\cdot)=0\big\}$.  The rest of the proof is standard using ``lifting''.

\smallskip
\noindent
  In \cite{bardos} the following type of approximation problem, in more general domains, was used
\begin{equation}\label{approxprobv0}
\Bigg\{
\begin{array}{cl}
\mu u^{\epsilon} + \mathcal{L}_{\epsilon}(u^{\epsilon}) - \epsilon\Delta u^{\epsilon} = f & \text{ in }\; \mathbb{R}^{d}_{+}\times\mathbb{S}^{d-1}\,,\\
u^{\epsilon} = 0 & \text{ on }\; \Gamma^{-}\,,\\
\partial_{x_{d}}u^{\epsilon} = 0 & \text{ on }\, \Gamma\,, 
\end{array}
\end{equation}
where $\mathcal{L}_{\epsilon}$, for $\epsilon>0$, is the operator defined as 
\begin{equation*}\label{approxop.}
\mathcal{L}_{\epsilon} := \theta\cdot\nabla_{x} - \mathcal{I}_{\epsilon}\,,
\end{equation*}
where $\mathcal{I}_{\epsilon}$ is the scattering operator defined by \eqref{2.5} where $(-\Delta_{\theta})^{s}$ in formula \eqref{2.4} is replaced by 
\begin{equation*}
\big[ (-\Delta^{\epsilon}_{\theta})^{s}u \big]_{\mathcal{J}}\defi \pro{\cdot}^{d-1+2s}\,(-\Delta^{\epsilon}_{v})^{s}U_{\mathcal{J}},\quad \text{with}\quad (-\Delta^{\epsilon}_{v})^{s}:= \frac{1}{1 + \epsilon\langle v \rangle^{2s}}(-\tilde\Delta^{\epsilon}_{v})^{s} \frac{1}{1 + \epsilon\langle v \rangle^{2s}}\,.
\end{equation*}
The operator $(-\tilde\Delta^{\epsilon}_{v})^{s}$ is a bounded approximation of the fractional Laplacian defined through its Fourier transform
\begin{equation*}
\mathcal{F}_{v}\{(-\tilde\Delta^{\epsilon}_{v})^{s}\varphi\}(\xi) := \frac{|\xi|^{2s}}{1+\epsilon\langle \xi \rangle^{2s}}\widehat{\varphi}(\xi)\,,\qquad \xi\in\mathbb{R}^{d-1}\,.
\end{equation*}
For this approximation it is easy to see that $\| \mathcal{I}_{\epsilon}\|_{\mathcal{B}(L^{2}_{x,\theta})} \leq C\epsilon^{-3}$, where $\|\cdot\|_{\mathcal{B}}$ is the operator norm.

\smallskip
\noindent
Given $f\in H^{1}_{x}\cap L^{2}_{\theta}$, problem \eqref{approxprobv0} has a unique weak solution $u^{\epsilon}\in H^{1}_{x}\cap L^{2}_{\theta}\cap\{\gamma^{-}(\cdot)=0\}$.  Furthermore, the normal gradient $\partial_{x_{d}}u^{\epsilon}$ is forced to vanish in the boundary, then using the Green's formula for $u^{\epsilon}$ and $\partial_{x_{d}}u^{\epsilon}$ one is led to the energy estimate
\begin{align*}
c\,\Big\| (-\tilde\Delta^{\epsilon}_v)^{s/2}\Big( \frac{U^{\epsilon}_{\mathcal{J}}}{1+\epsilon\langle v \rangle^{2s}} \Big)\Big\|^{2}_{L^{2}_{x,v}} + (\mu-C)\norm{u^{\epsilon}}^{2}_{ H^{1}_{x}\cap L^{2}_{\theta}}  \leq \|f\|^{2}_{H^{1}_{x}\cap L^{2}_{\theta}}\,.
\end{align*}
This estimate, with $\mu>C$, is enough to send $\epsilon\rightarrow 0$ and conclude that the sequence $\{u^{\epsilon}\}$ converges, up to a subsequence, weakly in $H^{1}_{x}\cap L^{2}_{\theta}$ to a weak solution $u$ of the problem
\begin{equation}\label{limitprob.}
\bigg\{
\begin{array}{cl}
\mu u + \mathcal{L}(u) = f  & \text{ in }\; \mathbb{R}^{d}_{+}\times\mathbb{S}^{d-1}\,,\\
u = 0 & \text{ on }\; \Gamma^{-}\,,\\
\end{array}
\end{equation}
satisfying the estimate
\begin{align*}
c\,\Big\| (-\Delta_v)^{s/2} U_{\mathcal{J}} \Big\|^{2}_{L^{2}_{x,v}}+(\mu-C)\norm{u}^{2}_{ H^{1}_{x}\cap L^{2}_{\theta}} \leq \|f\|^{2}_{H^{1}_{x}\cap L^{2}_{\theta}}\,.
\end{align*}
In other words, $u\in H^{1}_{x} \cap H^{s}_{\mathcal{L}}\cap\{\gamma^{-}(\cdot)=0\}$.

\smallskip
\noindent
We are now ready to conclude.  Take $\tilde{u}\in H^{s}_{\mathcal{L}}\cap\{\gamma^{-}(\cdot)=0\}$, then $\tilde{u}$ satisfies problem \eqref{limitprob.} with $f$ replaced by some $F\in L^{2}_{x,\theta}$.   Take a sequence $\{f^{n}\}\subset H^{1}_{x}\cap L^{2}_{\theta}$ converging strongly to $F$ in $L^{2}_{x,\theta}$.  By the previous argument, there exists a sequence $\{u^{n}\}\in H^{1}_{x} \cap H^{s}_{\mathcal{L}}\cap\{\gamma^{-}(\cdot)=0\}$ solving problem \eqref{limitprob.} with $f$ replaced by $f^{n}$.  Since the Green's formula is valid for such sequence, it is not difficult to conclude that
\begin{equation*}
\| u^{n} - u^{m}\|_{H^{s}_{\mathcal{L}}}\leq C\| f^{n} - f^{m}\|_{L^{2}_{x,\theta}}\,.
\end{equation*}
Thus, $\{u^{n}\}$ is Cauchy and converges strongly to a limit, say $\tilde{w}\in H^{s}_{\mathcal{L}}\cap\{\gamma^{-}(\cdot)=0\}$.  Indeed, recall that the trace is continuous from estimate \eqref{contrace}, consequently, $\tilde{w}$ has a well defined trace with $\gamma^{-}(\tilde{w})=0$.   And, of course,
\begin{equation*}
\| u^{n} - \tilde{w}\|_{H^{s}_{\mathcal{L}}}\leq C\| f^{n} - F\|_{L^{2}_{x,\theta}}\rightarrow 0\,\;\text{ as }\;n\rightarrow0\,.
\end{equation*}
In this way, $\tilde{w}$ is a weak solution of problem \eqref{limitprob.} with $f$ replaced by $F\in L^{2}_{x,\theta}$.  But, problem \eqref{limitprob.} has a unique weak solution, therefore $\tilde{w} = \tilde{u}$.  Consequently, $$\tilde{u}\in\overline{H^{1}_{x} \cap H^{s}_{\mathcal{L}}\cap\{\gamma^{-}(\cdot)=0\} }\,,$$ which proves the result.
\end{proof}
\begin{rem}
Using the Green's formula, it readily follows that if $u\in \big\{ w\in H^{s}_{\mathcal{L}} \, \big| \, \gamma^{-}(w)\in L^{2}\big(\Gamma^{-},|\theta_{d}|\text{d}\theta\text{d}\bar{x}\big) \big\}$, then $\gamma^{+}(u)\in L^{2}\big(\Gamma^{+},|\theta_{d}|\text{d}\theta\text{d}\bar{x}\big)$.
\end{rem}

\smallskip
\noindent
As a corollary, Theorem \ref{Tracetheorem} gives the Green's identity on compact sets of $\Gamma^{\pm}$.
\begin{cor}\label{weakgreen}
Let $\varphi(x,\theta)\in H^{1}_{x}\cap H^{s}_{\theta}$ with support in $\{\theta_{d}<0\}\cup\{\theta_{d}>0\}$.  Then for $u\in H^{s}_{\Gamma_{loc}}$ it follows that
\begin{align*}
\int_{\mathbb{R}^{d}_{+}}\int_{\mathbb{S}^{d-1}}\mathcal{L}(u)\,\varphi\, {\rm d}\theta {\rm d}x & = -\int_{\mathbb{R}^{d}_{+}}\int_{\mathbb{S}^{d-1}}\mathcal{I}(u)\, \varphi \, {\rm d}\theta {\rm d}x - \int_{\mathbb{R}^{d}_{+}}\int_{\mathbb{S}^{d-1}} u\, \theta\cdot\nabla_{x}\varphi\, {\rm d}\theta {\rm d}x\\
&\qquad- \int_{\Gamma^{+}}\gamma^{+}(u)\varphi(\bar{x},\theta)\, \theta_{d}\, {\rm d}\theta {\rm d}\bar{x} - \int_{\Gamma^{-}}\gamma^{-}(u)\varphi(\bar{x},\theta)\, \theta_{d} \,{\rm d}\theta {\rm d}\bar{x}\,.
\end{align*}
\end{cor}
\subsection{Absorbing boundary conditions.} We consider in this section the RTE with absorbing boundary, that is, the case where the input boundary intensity is null $g\equiv0$.  Non homogeneous boundary conditions follows from this case using classical arguments involving ``lifting'', refer for example to the classical reference \cite[Chapter XXI - section 4]{libro}.  In the case of absorbing boundary conditions, the natural domain of the operator $\mathcal{L}:\mathcal{D}(\mathcal{L})\rightarrow L^{2}_{x,\theta}$ is $$\mathcal{D}(\mathcal{L})= H^{s}_{\mathcal{L}}\cap\big\{ \gamma^{-}(\cdot)=0 \big\}\,.$$ 
In $\mathcal{D}(\mathcal{L})$ the Green's identity holds in the whole boundary since the traces belong to $L^{2}\big(\Gamma^{\pm},| \theta_{d} |\,{ \rm d}\theta\, {\rm d}\bar{x}\big)$.  Therefore, for any $u \in \mathcal{D}(\mathcal{L})$
\begin{align*}
\int_{\mathbb{R}^{d}_{+}}\int_{\mathbb{S}^{d-1}}\mathcal{L}(u)\,u \,{\rm d}\theta {\rm d}x & = -\int_{\mathbb{R}^{d}_{+}}\int_{\mathbb{S}^{d-1}}\mathcal{I}(u)\,u \, {\rm d}\theta {\rm d}x - \tfrac{1}{2}\int_{\Gamma^{+}} \big| \gamma^{+}(u) \big|^{2} \theta_{d}\, {\rm d}\theta {\rm d}\bar{x}\geq0\,.
\end{align*}
We conclude that $\langle \mathcal{L}(u),u\rangle_{L^{2}_{x,\theta}}\geq0$ and, consequently, $(-\mathcal{L})$ is a dissipative operator taking values on $L^{2}_{x,\theta}$.

\smallskip
\noindent
In order to apply Lumer--Phillips theorem it suffices to prove that the operator $(-\mathcal{L}) - \mu\,I$ is surjective on $L^{2}_{x,\theta}$ for some $\mu>0$.  But, this is exactly what we proved when we showed the existence of weak solutions in $\mathcal{D}(\mathcal{L})$ for the problem \eqref{limitprob.} with $f\in L^{2}_{x,\theta}$.  Thus, invoking the Lumer--Phillips theorem in reflexive spaces one proves the following theorem. 
\begin{thm}\label{contraction-semigroup}
The operator $(-\mathcal{L}):\mathcal{D}(\mathcal{L})\rightarrow L^{2}_{x,\theta}$ generates a contraction semigroup.
\end{thm}  
\begin{cor}[Absorbing boundary]\label{existence-absorbing}
Consider $u_0\in L^{2}\big(\mathbb{R}^{d}\times\mathbb{S}^{d-1}\big)$ and $f\in L^{2}\big((0,T)\times\mathbb{R}^{d}\times\mathbb{S}^{d-1}\big)$.  Then, the problem
\begin{equation*}
\left\{
\begin{array}{cll}
\partial_t u + \mathcal{L}(u) = f & \text{in} & (0,T)\times \mathbb{R}^{d}_+\times \mathbb{S}^{d-1},\\
u=u_0 & \text{on} & \{t=0\}\times \mathbb{R}^{d}_+\times \mathbb{S}^{d-1},\\
u=0  & \text{on} & (0,T) \times \partial  \mathbb{R}^{d}_+\times \mathbb{S}^{d-1} \;\;  \text{and}\;\; -(\theta \cdot n(\bar{x}))>0,
\end{array}\right.
\end{equation*}
has a unique weak solution $u(t)\in\mathcal{C}\big([0,T], L^{2}(\mathbb{R}^{d}_{+}\times\mathbb{S}^{d-1})\big)$.  Its trace satisfies $\gamma^{+}(u)\in L^{2}\big((0,T)\times\Gamma^{+}, |\theta_{d}| {\rm d}\theta {\rm d}\bar{x}\big)$.  Furthermore, $u(t)\geq0$ if $u_0\geq0$ and $f\geq0$.
\end{cor}
\begin{proof}
The statement about existence is direct from Theorem \ref{contraction-semigroup}.  Uniqueness follows from the fact that solutions with well defined traces in $L^{2}\big((0,T)\times\Gamma^{\pm}, |\theta_{d}| {\rm d}\theta {\rm d}\bar{x}\big)$ satisfy the Green's formula.

\smallskip
\noindent
Finally, in order to prove positivity we use estimate \eqref{eq2.4} and Gronwall's lemma to conclude that for any $t\in(0,T)$
\begin{equation*}
\begin{split}
\| u(t) \|^{2}_{L^{2}_{x,\theta}}& + D_0\int^{t}_{0}\int_{\R^{d}_{+}}\norm{u}^{2}_{H^{s}_{\theta}} \; \mathrm{d}x\; \mathrm{d}\tau \\
& + \tfrac{1}{2}\int^{t}_{0}\int_{\Gamma^{+}} | \gamma^{+}(u) |^{2}\, | \theta_{d} |\, \mathrm{d}\theta\; \mathrm{d}\bar{x}\; \mathrm{d}\tau \leq C(T)\Big( \| u_0 \|^{2}_{L^{2}_{x,\theta}} + \int^{t}_{0} \| f(\tau) \|^{2}_{L^{2}_{x,\theta}} \mathrm{d}\tau\Big). 
\end{split}
\end{equation*}
Now, for $s\in(0,1)$ it follows that if $u(x,\cdot)\in H^{s}_{\theta}$, then the positive and negative parts of $u$ satisfy $u^{\pm}\in H^{s}_{\theta}$ as well\footnote{In fact, one has $\|u^{\pm}\|_{H^{s}_{\theta}}\leq \|u\|_{H^{s}_{\theta}}$.}.  Then,
\begin{equation*}
\int_{\mathbb{S}^{d-1}} \mathcal{I}(u)u^{-}{\rm d}\theta = \int_{\mathbb{S}^{d-1}} \mathcal{I}(u^{+})u^{-}{\rm d}\theta + \int_{\mathbb{S}^{d-1}} \mathcal{I}(u^{-})u^{-}{\rm d}\theta \leq 0\,.
\end{equation*}
The fact that both terms in the left side are non positive is a direct consequence of the weak formulation \eqref{1.6}.  Applying the Green's formula again, it follows that for $t\in(0,T)$
\begin{equation*}
\begin{split}
\| u^{-}(t) \|^{2}_{L^{2}_{x,\theta}} + \tfrac{1}{2}\int^{t}_{0}\int_{\Gamma^{+}} | \gamma^{+}(u^{-}) |^{2}\, | \theta_{d} |\, \mathrm{d}\theta\; \mathrm{d}\bar{x}\; \mathrm{d}\tau \leq \| u^{-}_0\|^{2}_{L^{2}_{x,\theta}} + \int^{t}_{0}\int_{\mathbb{R}^{d}_{+}}\int_{\mathbb{S}^{d-1}} f\,u^{-} {\rm d} \theta {\rm d} x\mathrm{d}\tau. 
\end{split}
\end{equation*}
If $u_0\geq0$ and $f\geq0$, the left side is non positive.  One is led to conclude that $u^{-}\equiv0$.
\end{proof}

\noindent
\textbf{Acknowledgments}. R. Alonso gratefully acknowledge the support from Bolsa de Produtividade em Pesquisa CNPq. 
E. Cuba has been supported in part by Coordena\c{c}\~ao de Aperfei\c{c}oamento de Pessoal de N\'ivel Superior-Brasil (CAPES), PETROBRAS and PUC-Rio postgraduate grants.
\bigskip
\appendix
\section{Appendix}
\begin{prop}\label{app-inter-x-theta}
Let $s,s'>0$ and $\varphi\in H^{s}_{\theta} \cap H^{s'}_{x}$.  Then, for some $r:=r(s,s',d)>2$ and $\alpha:=\alpha(s,s',d)\in(0,1)$ it follows that
\begin{equation*}
\| \varphi \|_{L^{r}_{x,\theta}}\leq C\bigg(\int_{\R^{d}}\| \varphi(x,\cdot) \|^{2}_{H^{s}_{\theta}}{\rm d}x\bigg)^{\frac{\alpha}{2}}\bigg(\int_{\mathbb{S}^{d-1}}\big\|(-\Delta_x)^{s'/2} \varphi(\cdot, \theta) \big\|^{2}_{L^2_x}{\rm d}\theta\bigg)^{\frac{1-\alpha}{2}}\,.
\end{equation*} 
for some $C:=C(s,s',d)$
\end{prop}
\begin{proof}
Recall that by \eqref{2.14} and Sobolev imbedding we have that 
\begin{align*}
\int_{\R^d}\|\varphi(x, \cdot)\|_{H^s_\theta}^2 {\rm d }x &\geq c \int_{\R^d}\Big(\int_{\mathbb{S}^{d-1}} \big| \varphi(x, \theta) \big|^{p} {\rm d}\theta\Big)^{\frac{2}{p}}{\rm d}x\,, \\
\int_{\mathbb{S}^{d-1}}\|(-\Delta_x)^{s'/2} \varphi(\cdot, \theta)\|_{L^2_x}^2 {\rm d}\theta &\geq c \int_{\mathbb{S}^{d-1}}\left(\int_{\mathbb{R}^{d}} \big|\varphi(x, \theta)\big|^{q} {\rm d}x\right)^{\frac{2}{q}} {\rm }d\theta\,,
\end{align*}
where $c:=c(s,s',d)$ and
\begin{equation*} 
\frac{1}{q} = \frac 12 - \frac{s'}{d} \,, \qquad \frac{1}{p} = \frac 12 - \frac{s}{d-1} \,.
\end{equation*}
Observe that $q,\,p>2$.  Set
\begin{equation}\label{alpha12}
\alpha_1 = \frac{q - 2}{ \frac{p}{2} \, q - 2} \in (0, 1) \,, \qquad \alpha_2 = \frac{p}{2} \, \alpha_1 \in (0, 1) \,, \qquad  r = p \, \alpha_1 + 2(1 - \alpha_1) > 2 \,,
\end{equation} 
so that 
\begin{equation*}
\frac{\alpha_1}{\alpha_2} = \frac{2}{p} \,,\qquad \frac{1-\alpha_1}{1-\alpha_2} = \frac{q}{2}>1 \,.
\end{equation*}
Then, using H\"{o}der inequality we have that
\begin{align*}
\int_{\R^d} \int_{\mathbb{S}^{d-1}}& \big|\varphi(x,\theta) \big|^r {\rm d}\theta \, {\rm d}x\\
&\leq \int_{\R^d}\left(\int_{\mathbb{S}^{d-1}} \big|\varphi(x,\theta)\big|^{p} {\rm d}\theta\right)^{\alpha_1}\left(\int_{\mathbb{S}^{d-1}} \big|\varphi(x,\theta)\big|^{2} {\rm d}\theta\right)^{1-\alpha_1}{\rm d}x\\
&\leq\left(\int_{\R^d}\left(\int_{\mathbb{S}^{d-1}}\big| \varphi(x,\theta) \big|^{p} {\rm d}\theta\right)^{\frac{\alpha_1}{\alpha_2}}{\rm d}x\right)^{\alpha_2}\left(\int_{\R^{d}}\left(\int_{\mathbb{S}^{d-1}} \big|\varphi(x,\theta)\big|^{2} {\rm d}\theta\right)^{\frac{1-\alpha_1}{1-\alpha_{2}}}{\rm d}x\right)^{1-\alpha_{2}}\\
&=\left(\int_{\R^d}\left(\int_{\mathbb{S}^{d-1}} |\varphi(x,\theta)\big|^{p} {\rm d}\theta\right)^{\frac{2}{p}}{\rm d}x\right)^{\alpha_2}\left(\int_{\R^{d}}\left(\int_{\mathbb{S}^{d-1}} \big|\varphi(x,\theta)\big|^{2} {\rm d}\theta\right)^{\frac{q}{2}}{\rm d}x\right)^{1-\alpha_{2}}
\\
&\leq C \left(\int_{\R^d} \|\varphi(x,\cdot)\|_{H^s_\theta}^2 {\rm d}x\right)^{\alpha_2}\left(\int_{\mathbb{S}^{d-1}}\left(\int_{\R^{d}} \big|\varphi(x,\theta)\big|^{q} {\rm d}x\right)^{\frac{2}{q}}{\rm d}\theta\right)^{\frac{q}{2}(1-\alpha_{2})}\\
&\leq C \left(\int_{\R^d} \|\varphi(x,\cdot)\|_{H^s_\theta}^2 {\rm d}x\right)^{\alpha_2}\left(\int_{\mathbb{S}^{d-1}} \big\|(-\Delta_x)^{s'/2} \varphi(\cdot, \theta) \big\|^{2}_{L^2_x} {\rm d}\theta\right)^{\frac{q}{2}(1-\alpha_{2})}\,.
\end{align*}
We used Minkowski's integral inequality for the second term just after the equality.  With the definitions \eqref{alpha12} it is easy to check that
\begin{equation*}
2\alpha_{2} + q(1-\alpha_{2})=r\,,
\end{equation*} 
thus, the result follows setting $\alpha=\frac{2\alpha_{2}}{r}$.
\end{proof}

\begin{prop}\label{prop-app1}
The following estimate holds for any $l\geq0$ and $\epsilon>0$
\end{prop}
\begin{equation*}
\| (-\Delta_{x})^{l/2}(\phi'_{\epsilon}u)\|_{L^{2}_{x}}\leq \epsilon^{-1}C_{\phi} \| (-\Delta_{x})^{l/2}(\phi_{\epsilon/2}u)\|_{L^{2}_{x}} + \epsilon^{-1-l}C_{\phi} \|\phi_{\epsilon/2}u\|_{L^{2}_{x}}\,.
\end{equation*}
\begin{proof}
For $l\in\mathbb{N}$ the result is clear.  For a general $l>0$ one writes it as $l = \lfloor l \rfloor + \alpha$ with $\alpha\in(0,1)$.  In addition, observe that
\begin{equation*}
\phi'_{\epsilon}u = \epsilon^{-1}\phi'(x/\epsilon)u = \epsilon^{-1}\phi'(x/\epsilon) (\phi_{\epsilon/2}u)\,.
\end{equation*}
Therefore, with the notation $\psi_{\epsilon}(x)=\phi'(x/\epsilon)$, it follows that
\begin{equation*}
\begin{split}
(-\Delta_{x})^{\alpha/2}(\phi'_{\epsilon}u)& = \epsilon^{-1}\psi_{\epsilon}\, (-\Delta_{x})^{\alpha/2}(\phi_{\epsilon/2}u) \\
&\qquad + c_{d,s}\epsilon^{-1}\int_{\mathbb{R}^{d}}\frac{\psi_{\epsilon}(x) - \psi_{\epsilon}(x-y)}{|y|^{d+\alpha}}(\phi_{\epsilon/2}u)(x-y)\text{d}y\,.
\end{split}
\end{equation*}
For the last term in the right it follows that
\begin{equation*}
\begin{split}
\epsilon^{-1}\int_{\mathbb{R}^{d}}&\frac{\psi_{\epsilon}(x) - \psi_{\epsilon}(x-y)}{|y|^{d+\alpha}}(\phi_{\epsilon/2}u)(x-y)\text{d}y=\\
&\qquad\epsilon^{-1}\bigg(\int_{|y|\leq\epsilon}+\int_{|y|>\epsilon}\bigg)\frac{\psi_{\epsilon}(x) - \psi_{\epsilon}(x-y)}{|y|^{d+\alpha}}(\phi_{\epsilon/2}u)(x-y)\text{d}y\,.
\end{split}
\end{equation*}
For the latter integral one has that
\begin{equation*}
\bigg|\int_{|y|>\epsilon}\frac{\psi_{\epsilon}(x) - \psi_{\epsilon}(x-y)}{|y|^{d+\alpha}}(\phi_{\epsilon/2}u)(x-y)\text{d}y\bigg| \leq 2\|\phi'\|_{\infty} \int_{\R^{d}}\frac{\textbf{1}_{|y|>\epsilon}}{|y|^{d+\alpha}}\big| (\phi_{\epsilon/2}u)(x-y) \big|\text{d}y\,.
\end{equation*}
Since $\int |y|^{-d-\alpha}\textbf{1}_{|y|>\epsilon}\sim \epsilon^{-\alpha}$, the $L^{2}_{x}$-norm of this term is controlled by $\epsilon^{-\alpha}C_{\phi}\|\phi_{\epsilon/2}u\|_{2}$.  For the first integral one applies the fact that $\psi_{\epsilon}$ is Lipschitz with constant $\epsilon^{-1}\|\phi''\|_{\infty}$.  As a consequence, 
\begin{equation*}
\bigg|\int_{|y|\leq\epsilon}\frac{\psi_{\epsilon}(x) - \psi_{\epsilon}(x-y)}{|y|^{d+\alpha}}(\phi_{\epsilon/2}u)(x-y)\text{d}y\bigg| \leq \epsilon^{-1}\|\phi''\|_{\infty} \int_{\R^{d}}\frac{\textbf{1}_{|y|>\epsilon}}{|y|^{d+\alpha-1}}\big| (\phi_{\epsilon/2}u)(x-y) \big|\text{d}y\,.
\end{equation*}
Since $\int |y|^{-d-\alpha+1}\textbf{1}_{|y|\leq\epsilon}\sim \epsilon^{1-\alpha}$, the $L^{2}_{x}$-norm of this term is controlled by $\epsilon^{-\alpha}C_{\phi}\|\phi_{\epsilon/2}u\|_{2}$.  Overall these estimates prove that
\begin{equation}\label{app-e1}
\| (-\Delta_{x})^{\alpha/2}(\phi'_{\epsilon}u)\|_{L^{2}_{x}}\leq \epsilon^{-1}C_{\phi} \| (-\Delta_{x})^{\alpha/2}(\phi_{\epsilon/2}u)\|_{L^{2}_{x}} + \epsilon^{-1-\alpha}C_{\phi} \|\phi_{\epsilon/2}u\|_{L^{2}_{x}}\,.
\end{equation}
For the general case note that $(-\Delta_{x})^{l/2} = (-\Delta_{x})^{\alpha/2}(-\Delta_{x})^{\lfloor l \rfloor/2}$.  One can use the product rule to treat the operator $(-\Delta_{x})^{\lfloor l \rfloor/2}$, then use estimate \eqref{app-e1} and interpolation to conclude.
\end{proof}
\begin{prop}[Commutator]\label{prop-app2}
Fix dimension $d\geq2$, $\alpha\in(0,1)$, and $l, k\geq0$.  Then, for any suitable $\varphi$
\begin{equation}\label{commutator}
\big( 1-\Delta_{v} \big)^{\alpha/2}\big(\langle v \rangle^{2l}\varphi\big)= \langle v \rangle^{2l}\big( 1-\Delta_{v} \big)^{\alpha/2}\varphi + \mathcal{R}(\varphi)\,,
\end{equation}
where the operator $\mathcal{R}$ satisfies the estimate
\begin{equation}\label{commutator-rem}
\big\| \langle \cdot \rangle^{-2k}\mathcal{R}(\varphi) \big\|_{L^{2}_{v}}\leq C_{d,\alpha,l,k}\| \langle v \rangle^{\max\{0,2l -1\}-2k}\varphi \big\|_{L^{2}_{v}}\,.
\end{equation}
\end{prop}
\begin{proof}
This type of formulas are common in the literature, thus, we will be brief and leave the computational details to the reader. 

\smallskip
\noindent
Keep in mind the identity $\big( 1-\Delta_{v} \big)^{\alpha/2}=\big( 1-\Delta_{v} \big)\big( 1-\Delta_{v} \big)^{-(1-\alpha/2)}$.  Let $\mathcal{B}_{1-\alpha/2}$ be the Bessel kernel associated to $\big( 1-\Delta_{v} \big)^{-(1-\alpha/2)}$, then
\begin{align*}
\big( 1-\Delta_{v} \big)^{-(1-\alpha/2)}\big(\langle \cdot \rangle^{2l}\varphi\big) &= \langle v \rangle^{2l}\big( 1-\Delta_{v} \big)^{-(1-\alpha/2)}\varphi \\
& \qquad+ \int_{\mathbb{R}^{d}}\mathcal{B}_{1-\alpha/2}(v'-v)\big(\langle v' \rangle^{2l} - \langle v \rangle^{2l} \big)\varphi(v')\text{d}v'\,.
\end{align*}
We define
\begin{equation*}
\mathcal{R}_{1}(\varphi):=(1-\Delta_{v})\int_{\mathbb{R}^{d}}\mathcal{B}_{1-\alpha/2}(v'-v)\big(\langle v' \rangle^{2l} - \langle v \rangle^{2l} \big)\varphi(v')\text{d}v'\,,
\end{equation*}
and prove that satisfies estimate \eqref{commutator-rem}.  Using Taylor expansion, it follows that
\begin{equation*}
\langle v' \rangle^{2l} = \langle v + (v' - v) \rangle^{2l} = \langle v \rangle^{2l} + \nabla\langle v \rangle^{2l}\cdot(v' - v) + (v' - v)\cdot\mathcal{H}(v',v)(v' - v)\,,
\end{equation*}
where the remainder is given by
\begin{equation*}
\int^{1}_{0}(1-\tau)\nabla^{2}\langle \cdot \rangle^{2l}\big( \tau v' + (1-\tau)v \big)\text{d}\tau\,.
\end{equation*}
For the first order term in the Taylor expansion, it holds asymptotically  
\begin{equation*}
\mathcal{B}_{1-\alpha/2}(v'-v)(v'-v) \sim \big( 1 - \Delta_{v}\big)^{-1 - \frac{1 - \alpha}{2} }\,,\quad\text{for}\quad \ |v|\ll 1\,,
\end{equation*}
and consequently, 
\begin{equation*}
(1-\Delta_{v})\mathcal{B}_{1-\alpha/2}(v'-v)(v'-v) \sim \big( 1 - \Delta_{v}\big)^{-\frac{1-\alpha}{2} }\,,\quad\text{for}\quad \ |v|\ll 1\,.
\end{equation*}
Similarly, for the second order term one has the asymptotic behaviour
\begin{equation*}
\mathcal{B}_{1-\alpha/2}(v'-v) \times (v' - v)\cdot\mathcal{H}(v',v)(v' - v) \sim \big( 1 - \Delta_{v} \big)^{-2 + \frac{\alpha}{2} }\,,\quad\text{for}\quad |v|\ll 1\,,
\end{equation*}
and therefore, for $|v|\ll1$ it follows that
\begin{equation*}
(1 - \Delta_{v})\mathcal{B}_{1-\alpha/2}(v'-v) \times (v' - v)\cdot\mathcal{H}(v',v)(v' - v) \sim \big( 1 - \Delta_{v} \big)^{-1 + \frac{\alpha}{2} }\,.
\end{equation*}
Additionally and regarding the weight grow, note that
\begin{equation*}
\big| \nabla\langle v \rangle^{2l} \big| \sim \langle v \rangle^{ 2l-1}\,,\qquad  \big| \mathcal{H}(v',v) \big| \lesssim \langle v' \rangle^{ \max\{ 0 , 2l-2 \} }  + \langle v \rangle^{ \max\{ 0, 2l - 2\} }\,.
\end{equation*}
These observations together with classical asymptotic properties of Bessel kernels, see for instance \cite[Section 6.1.2]{Graf}, readily lead to \eqref{commutator-rem} for the operator $\mathcal{R}_{1}$.  The rest of the proof follows by the standard product rule of differentiation. 
\end{proof}

\noindent\textsc{Ricardo Alonso}\\
Departamento de Matem\'atica\\
Pontif\'icia Universidade Cat\'olica do Rio de Janeiro -- PUC-Rio\\
22451-900, G\'avea, Rio de Janeiro-RJ, Brazil\\
\noindent\texttt{ralonso@mat.puc-rio.br}

\bigskip

\noindent\textsc{Edison Cuba}\\
Departamento de Matem\'atica\\
Pontif\'icia Universidade Cat\'olica do Rio de Janeiro -- PUC-Rio\\
22451-900, G\'avea, Rio de Janeiro-RJ, Brazil\\
\noindent\texttt{edisonfausto@mat.puc-rio.br}

\end{document}